\documentclass[10pt,reqno]{article}
\usepackage[letterpaper,margin=1in,footskip=0.30in]{geometry}
\usepackage{microtype}
\usepackage{amsmath,amssymb,amsthm}
\usepackage{mathtools}
\usepackage[all]{xy}
\makeatletter
  \newcommand{\supsize}{%
    \expandafter\ifx\csname S@\f@size\endcsname\relax
      \calculate@math@sizes
    \fi
    \csname S@\f@size\endcsname
    \fontsize\sf@size\z@\selectfont
  }
  \DeclareRobustCommand{\tsup}[1]{%
    \leavevmode\raise.9ex\hbox{\supsize #1}%
  }
  \DeclareTextSymbolDefault{\textprimechar}{OMS}
  \DeclareTextSymbol{\textprimechar}{OMS}{48}
  \DeclareRobustCommand{\tprime}{\tsup{\textprimechar}}
  \ProvideTextCommandDefault{\cprime}{\tprime}
\makeatother

\usepackage[dvipsnames]{xcolor}

\usepackage{enumitem}
\setlist{noitemsep}
\usepackage{mathrsfs}
\usepackage{tikz-cd}
\usepackage{relsize}
\usepackage[bbgreekl]{mathbbol}
\usepackage{amsfonts}
\usepackage[pdfusetitle,colorlinks]{hyperref}
\hypersetup{citecolor=Black,linkcolor=Black,urlcolor=Black}
\usepackage[capitalise,noabbrev]{cleveref}
\crefformat{equation}{\ensuremath{(#2#1#3)}}
\crefmultiformat{equation}{\ensuremath{(#2#1#3)}}{ and~\ensuremath{(#2#1#3)}}{, \ensuremath{(#2#1#3)}}{, and~\ensuremath{(#2#1#3)}}

\numberwithin{figure}{section}
\numberwithin{equation}{section}

\newtheorem{theorem}[figure]{Theorem}
\newtheorem{lemma}[figure]{Lemma}
\newtheorem{corollary}[figure]{Corollary}
\newtheorem{proposition}[figure]{Proposition}
\theoremstyle{definition}
\newtheorem{definition}[figure]{Definition}

\theoremstyle{definition}
\newtheorem{remark}[figure]{Remark}
\theoremstyle{definition}
\newtheorem{example}[figure]{Example}

\newtheoremstyle{cited}{.5\baselineskip\@plus.2\baselineskip\@minus.2\baselineskip}{.5\baselineskip\@plus.2\baselineskip\@minus.2\baselineskip}{\itshape}{}{\bfseries}{\bfseries .}{5pt plus 1pt minus 1pt}{\thmname{#1}\thmnumber{ #2}\thmnote{ \normalfont #3}}
\theoremstyle{cited}

\newcommand{\w}{\text}

\DeclareSymbolFontAlphabet{\mathbb}{AMSb} 
\DeclareSymbolFontAlphabet{\mathbbl}{bbold}
\newcommand{\Prism}{{\mathlarger{\mathbbl{\Delta}}}}

\newcommand{\cosimp}[3]{\xymatrix@1{#1 \ar@<.4ex>[r] \ar@<-.4ex>[r] & {\ }#2 \ar@<0.8ex>[r] \ar[r] \ar@<-.8ex>[r] & {\ } #3  \cdots }} 

\title{{Dieudonn\'e Theory via Cohomology of Classifying Stacks}}
\author{Shubhodip Mondal}
\date{}

\begin{document}

\maketitle

\begin{abstract}We prove that if $G$ is a finite flat group scheme of $p$ power rank over a perfect field of characteristic $p$, then the second crystalline cohomology of its classifying stack $H^2_{crys}(BG)$ recovers the Dieudonn\'e module of $G$. We also provide a calculation of crystalline cohomology of classifying stack of abelian varieties. We use this to prove that crystalline cohomology of the classifying stack of a $p$-divisible group is a symmetric algebra (in degree $2$) on its Dieudonn\'e module. We also prove mixed characteristic analogues of some of these results using prismatic cohomology. 
    
\end{abstract}
\tableofcontents
\newpage

\section{Introduction}
 Let $p$ be a prime which will be fixed throughout. Our beginning point is the classification of finite commutative group schemes of $p$ power rank over a perfect field $k$ of characteristic $p$ in terms of contravariant Dieudonn\'e theory. This is achieved by defining an algebraic invariant associated to any given finite group scheme $G$ (assumed to be commutative and $p$ power rank unless otherwise mentioned) which is called the Dieudonn\'e module of $G$ and will be denoted by $M(G).$ We will recall the classical construction of $M(G)$ in \cref{secc3}. The basic object of interest is the Dieudonn\'e ring $\mathscr{D}_k$ which is an associative algebra over the ring of Witt vectors $W(k).$ The Dieudonn\'e functor $M$ which takes $G$ to $M(G)$ transforms the study of finite group schemes over $k$ to the study of semi-linear algebraic objects. More precisely,

\begin{theorem}[Dieudonn\'e]\label{1}The Dieudonn\'e functor $M$ induces an anti-equivalence between the category of finite group schemes over $k$ and left $\mathscr{D}_k$-modules with finite $W(k)$ length.

\end{theorem}{}

Our primary goal in this paper is to provide a geometric construction of $M(G).$ In particular, we attempt to understand if $M(G)$ can be constructed in terms of some \textit{cohomology} theory associated to some ``\textit{space.}'' Since $M(G)$ is a module over $W(k)$, it could be natural to take the naive candidate, which is the $i$-th crystalline cohomology of $G$ for some $i$ denoted as $H^i_{crys}(G)$. However, one can see that this cannot not work for any $i$ simply by taking discrete group schemes as input for $M(.)$. Moreover, any reasonable cohomology theory will suffer from exactly similar problems. The issue is that a cohomology theory will view any discrete group scheme just as a discrete space and will be unable to detect its group structure. This suggests that we look for other spaces naturally associated to $G$ instead of just $G$ itself. In topology, one can get around this issue by considering the classifying space $BG$ as explained below in \cref{lol90}. In category theory, the analogue of this would be to simply view a group $G$ as a groupoid with one object and morphism set isomorphic to $G.$ In algebraic geometry, for a group scheme $G$, a model of these constructions would be to consider the classifying \textit{stack} $BG.$ Our main goal is to prove the following

\begin{theorem}\label{2}For a finite group scheme $G$ over $k$ which is of $p$ power rank, the Dieudonn\'e functor $M(G)$ is naturally isomorphic (upto a Frobenius twist) to $ H^2 _{crys}(BG)$, where $BG$ denotes the classifying stack of $G.$

\end{theorem}{}

\begin{example}\label{lol90}We explain the topological analogue of \Cref{1} and \Cref{2}. In the case where $G$ is a finite abelian group, one can consider the classifying space $BG$ of $G$ which has the property that $\pi_1(BG) \simeq G$ and $\pi_i(BG) = 0$ for $i>1.$ In this case, $H^2_{\w{singular}}(BG) \simeq \w{Ext}^1 (G, \mathbb{Z}),$ which is non-canonically isomorphic to the group $G.$ By using the short exact sequence $0 \to \mathbb{Z}\to  \mathbb{Q} \to \mathbb{Q}/ \mathbb{Z}\to 0$ we obtain that $ \w{Ext}^1(G,\mathbb{Z}) \simeq \w{Hom}(G,\mathbb{Q}/ \mathbb{Z} ) \simeq \varinjlim \w{Hom}(G, \mathbb{Z}/n  \mathbb{Z})\simeq \w{Hom}(G, \mathbb{S}^1).$ By Pontryagin duality, one obtains that sending $G \to H^2_{\w{singular}}(BG)$ gives an anti-equivalence from the category of finite abelian groups to itself.
\end{example}

We also define the classifying stack $BG$ of a $p$-divisible group $G$ and prove the following

\begin{theorem}\label{alvida}Let $G$ be a $p$-divisible group over $k.$ Then $H^*_{crys}(BG) \simeq \w{Sym}^{2*} M(G),$ where $M(G)$ denotes the Dieudonn\'e module of $G.$ In particular, $H^2_{crys}(BG)$ recovers the Dieudonn\'e module of $G.$

\end{theorem}{}

In \cite{SW14}, Scholze and Weinstein defined a mixed characteristic analogue of (covariant) Dieudonn\'e modules for $p$-divisible groups over a perfectoid ring. More recently, in \cite{AL19}, Ansch\"utz and Le Bras defined a mixed characteristic analogue of contraviariant Dieudonn\'e modules over more general base rings. In \cref{secc4.2} we briefly recall their work and the definition of the filtered prismatic Dieudonn\'e module $\underline{M}_{\Prism}(G)=( M_{\Prism}(G),\w{Fil}M_{\Prism}(G),\varphi_{M_{\Prism}(G)}  )$ associated to a $p$-divisible group $G$ as defined by them. Their main theorem is the following classification result.

\begin{theorem}[{\cite{AL19}}]Let $R$ be a quasi-regular semiperfectoid ring. The filtered prismatic Dieudonn\'e module functor $$G \to \underline{M}_{\Prism}(G) $$ defines an anti-equivalence between the category $p$-divisible groups over $R$ and the category of filtered prismatic Dieudonn\'e modules over $R.$

\end{theorem}{}

Using the classifying stack $BG$ of a $p$-divisible group $G$ and prove the following 
\begin{theorem}\label{thmpri}Let $G$ be a $p$-divisible group over a quasi-regular semiperfectoid ring $R.$ Then the prismatic cohomology  $H^2_{\Prism}(BG)$ is naturally isomorphic to the prismatic Dieudonn\'e module ${M}_{\Prism}(G).$ Further, the Nygaard filtration $\mathcal{N}^{\ge 1} H^2 _{\Prism}(BG)$ on prismatic cohomology of the stack $BG$ is isomorphic to $\w{Fil} M_{\Prism} (G).$ 

\end{theorem}{}

On the other hand, there has been some recent work involving computation of $p$-adic cohomology theories associated to classifying stack of group schemes over $p$-adic base rings. For example, we refer to the work of Totaro \cite{Tot18} and the work of Antieau, Bhatt and Mathew \cite{ABM19}. Our results can also be viewed as some computations in this direction.\\

\noindent
\textbf{Outline.}
\noindent
The statement of \Cref{2} presents the need for having a theory of crystalline cohomology for stacks. This theory has already been developed by Olsson in \cite{Ol07} using the lisse-\'etale crystalline site of an algebraic stack. In \Cref{sec2}, we provide another definition of crystalline cohomology of algebraic stacks through syntomic descent relying on the work of Fontaine and Messing \cite{FM87}. In \cref{comp} we prove that the definition via syntomic descent and the definition via the lisse-\'etale crystalline site are equivalent. \\

In \cref{3.1}, we provide a brief review of crystalline Dieudonn\'e theory and prove \cref{2}. Our proof uses a description of $M(G)$ obtained in the work of Berthelot, Messing and Breen \cite{BBM} which expresses $M(G)$ as a certain Ext group in the crystalline topos which is recalled in \cref{gen}. The main ingredient of our proof of \cref{2} is to obtain a suitable spectral sequence computing $H^*_{crys}(BG)$ which is done in \cref{lol3}. In order to obtain the spectral sequence in \cref{lol3}, we use \v{C}ech descent along the syntomic map $* \to BG.$ In \cref{keno}, we deduce the analogue of \cref{2} for $p$-divisible groups. In \cref{ab}, we provide a complete calculation of crystalline cohomology of the classifying stack $B A$ of an abelian variety $A.$ This is done in \cref{cw}, where we prove that $H^{2*}_{crys}(BA) \cong \w{Sym}^*(H^1_{crys}(A))$ and $H^i_{crys}(BA) = 0$ for odd $i.$ An analogue of this result in $\ell$-adic cohomology was proved by Behrend {\cite[Thm. \ 6.1.6.]{Beh03}}. Our proof of \cref{cw} uses different techniques. We rely on some explicit simplicial constructions and some computations from the theory of non-abelian derived functors. Using these calculations, in \cref{man}, we completely calculate the crystalline cohomology of $BG$ for a $p$-disivible group $G.$ \\

In \cref{sec4}, we enter the mixed characteristic situation. We begin by recalling the work of  Ansch\"utz and Le Bras and define the contravariant prismatic Dieudonn\'e module $M_{\Prism}(G)$  for a $p$-divisible group $G$ over a quasi-regular semiperfectoid base ring $R$  as an $\w{Ext}^1$ group in the prismatic topos. This definition is compatible with the definition in \cite{SW14} as proven in {\cite[Prop.\ 4.3.7.]{AL19}}. We will then define quasi-syntomic stacks and define the classifying stack $BG$ of $G$ as a quasi-syntomic stack in \cref{pdiv}. In \cref{keno1}, we extend the notion of prismatic cohomology developed by Bhatt and Scholze \cite{BS19} to quasi-syntomic stacks via quasi-syntomic descent. Then analogous to \cref{lol6}, in \cref{pri7} we prove that $M_{\Prism}(G) \simeq H^2_{\Prism}(BG),$ where the latter denotes the prismatic cohomology of $BG.$ Similar to the crystalline case, in \cref{pri10} we obtain a complete calculation of prismatic cohomology of the classifying stack $B \widehat{A}$, where $\widehat{A}$ is the $p$-adic completion of an abelian scheme $A.$\\

\noindent
\textbf{Acknowledgements.} I am grateful to Bhargav Bhatt for suggesting that Dieudonn\'e modules could be recoverable via crystalline cohomology of classifying stacks as well as numerous helpful conversations regarding various aspects of this paper. I would also like to thank Igor Kriz, Arthur-C\'esar Le bras and Martin Olsson for helpful conversations.

\newpage

\section{Crystalline cohomology}\label{sec2}
\subsection{Crystalline cohomology for stacks}
We gather some notations for this section. We fix a prime $p$. Let $k$ be a perfect field of characteristic $p> 0$ and let $W(k)$ be the ring of Witt vectors of $k.$ Let $\w{SYNSch}_k$ denote the big Grothendieck site of schemes over $k$ with the syntomic topology. Let $D (\w{SYNSch}_k, W(k))$ denote the derived category of sheaves of $W(k)$-modules and let $D(W(k))$ denote the derived category of $W(k)$-modules. For a given scheme $X \in \w{SYNSch}_k$, there is a functor $R \Gamma(X, \cdot):D (\w{SYNSch}_k, W(k))\to D(W(k)).$ Let $R \Gamma_{crys}(X) \in D(W(k))$ denote the crystalline cohomology of $X$.

\begin{proposition}\label{fonmes} There is an object $R \mathcal{O}^{crys} \in D (\w{SYNSch}_k, W(k))$ such that $R \Gamma(X, R \mathcal{O}^{crys}) \simeq R\Gamma_{crys}(X)$.
\end{proposition}

\begin{proof} This is a formal corollary of a result of Fontaine and Messing. Using \cite[Prop. 1.3.]{FM87}, we get that for any $n$, there is a sheaf $\mathcal{O}^{crys}/W_n$ on $\text{SYNSch}_k$ such that the $n$-truncated crystalline cohomology $(R \Gamma_{crys}/W_n)(X) \simeq R \Gamma(X, \mathcal{O}^{crys}/W_n)$. Now we define $R\mathcal{O}^{crys}:= R \varprojlim _{n} \mathcal{O}^{crys}/W_n.$ Then it follows that  $R \Gamma(X, R \mathcal{O}^{crys}) \simeq \varprojlim_{n} (R \Gamma_{crys}/W_n)(X) \simeq R \Gamma_{crys} (X). $\end{proof}

\begin{definition}\label{def2}Let $\mathscr X$ be an algebraic stack over $k$. Let $C_{\mathscr X}$ denote the category of schemes $X$ over $\mathscr X$. We define the \textit{crystalline cohomology of the algebraic stack} $\mathscr X$ as $$R \Gamma_{crys} (\mathscr X):= R\varprojlim _{X \in C_{\mathscr X}} R \Gamma_{crys} (X) \in D(W(k)).$$

\end{definition}{}
\begin{remark}Alternatively, given an algebraic stack $\mathscr{X}$ one can consider the slice site $\mathscr{X}_{syn}:=\w{SYNSch}_{k_{/\mathscr{X}}}.$ Then the object $R \mathcal{O}^{crys}$ induces an object in the derived category of sheaves of $W(k)$-modules on $\mathscr{X}_{syn}$ which by abuse of notation we again write as $R \mathcal{O}^{crys}.$ Then $R \Gamma_{crys}(\mathscr{X}) \simeq R \Gamma (\mathscr{X}_{syn}, R\mathcal{O}^{crys}).$

\end{remark}{}

 \begin{remark}\label{aff}By \Cref{def2} and Zariski descent, we observe that $R \Gamma_{crys}(\mathscr{X}) \simeq  R\varprojlim _{X \in C_{\mathscr X}^{\w{aff}}} R \Gamma_{crys} (X) $ where $C_\mathscr{X}^{\w{aff}}$ denotes the category of affine schemes over $\mathscr{X}.$
 \end{remark}
 


\begin{remark}\label{desc}
Let $X \to \mathscr{X}$ be a syntomic cover of the algebraic stack $\mathscr X$ by a scheme $X.$ Let $X_\bullet$ be the simplicial algebraic space given by the \v{C}ech nerve of $X \to \mathscr{X}$ (so that $X_n$ is the $n$-fold fibre product of $X$ over $\mathscr X$, which is, in general, only an algebraic space.) Then we have that $R \Gamma _{crys}(\mathscr{X}) \simeq R\varprojlim R \Gamma_{crys}(X_\bullet)$.  
\end{remark}{}

\subsection{The crystalline site} In this section, we provide a brief reminder of definitions of the (big) crystalline site and the associated topos. We prefer to work with the big site because it is functorial and still computes the same cohomology groups. Our exposition roughly follows \cite{BBM}.\\

We fix a prime $p$ as before. Let $k$ be a fixed field of characteristic $p$ and let $W(k)$ be the ring of Witt vectors. Let $(W(k), p, \gamma)$ be the usual divided power structure on $W(k).$

\begin{definition}\label{crys1}
We denote by $\w{Crys}(k/ W(k))$ a category whose objects are given by the data of an $k$-scheme $U$; a $W(k)$-scheme $T$ on which $p$ is locally nilpotent; a closed $W(k)$-immerison $i : U \to T$ and a divided power structure $\delta$ on the ideal sheaf corresponding to the closed immersion $i,$ which we require to be compatible with the divided power structure $\gamma.$ We will denote an object of $\w{Crys}(k/ W(k))$ by $(U,T, \delta)$ or simply by $(U,T)$ when no confusion is likely to arise. The morphisms of this category is supposed to be a pair of morphisms $u: U' \to U$ and $v: T' \to T$, where $u$ is an $k$-morphism and $v$ is a $W(k)$-morphism compatible with divided powers such that $v \circ i' = i \circ u.$
\end{definition}

\begin{definition}\label{crys2}A family $(U_i, T_i) \to (U,T)$ of maps in $\w{Crys}(k/ W(k))$ is a $\tau$-covering if $U_i= U \times_T T_i$ for all $i$ and $\left \{T_i \to T \right \}$ is a $\tau$-covering. Here $\tau$ could be Zariski, \'etale, smooth, syntomic or fppf. This equips the category $\w{Crys}(k/ W(k))$ with a Grothendieck topology and we denote the site we obtain this way by $\w{Crys}(k/ W(k))_{\tau}.$ we let $(k/ W(k))_{\w{Crys}, \tau}$ denote the associated topos. We define a presheaf $\mathcal{O}^{crys}(U,T) := \Gamma(T, \mathcal O_T).$ This is a sheaf of rings on the site $\w{Crys}(k/ W(k))_{\tau}.$   \\
\end{definition}

\begin{definition}\label{crys3}
Let $X$ be a scheme over $k$. By $\underline{X}$ we denote a sheaf on $\w{Crys}(k/ W(k))_{\tau}$ defined by $\underline{X}(U,T) := X(U) = \w{Hom}_{k}(U,X).$ \textit{cf.} {\cite[1.1.4.5.]{BBM}} This is a sheaf for any $\tau.$
\end{definition}

\begin{remark}We will write $\tau \le \tau'$ while comparing different topologies to mean that $\tau'$ is finer than $\tau.$

\end{remark}{}

\begin{proposition}\label{crys4} Let $G$ be an abelian sheaf on $\w{Crys}(k/ W(k))_{\tau'}.$ Then $$\w{RHom}_{\w{Crys}(k/ W(k))_{\tau}}(G, \mathcal O^{crys}) \simeq \w{RHom}_{\w{Crys}(k/ W(k))_{\tau'}}(G, \mathcal O^{crys})  $$ for $\tau \le \tau'.$ 

\end{proposition}{}

\begin{proof} This is {\cite[Prop.\ 1.3.6. ]{BBM}}. The proof relies on {\cite[Prop. 1.1.19.]{BBM}}\end{proof}

\begin{proposition}\label{why} Let $X$ be a scheme over $k$. Then $$R\Gamma_{crys}(X) \simeq \w{RHom}_{\w{Crys}(k/ W(k))_{\tau}}(\mathbb Z [\underline{X}], \mathcal O^{crys}),$$ where one can choose any topology $\tau.$ Here $\mathbb Z[\underline{X}] $ denotes the ``free abelian group" on $\underline{X}$ in the topos $(k/ W(k))_{\w{Crys}, \tau}.$

\end{proposition}{}

\begin{proof}When $\tau$ is Zariski topology, this follows from {\cite[Prop. \ 1.3.4.]{BBM}} and by definition of crystalline cohomology.  The rest follows from \Cref{crys4}. \end{proof}{}

\begin{remark}
Therefore, we see that it does not matter which topology $\tau$ we choose while computing crystalline cohomology. We will often choose $\tau$= Zariski topology and in this case we will omit $\tau$ from the notation and write the relevant site as $\w{Crys}(k/ W(k))$ and the topos as $(k/ W(k))_{\w{Crys}}.$
\end{remark}{}

\begin{remark}\label{trun}
One can also define the ``truncated crystalline sites'' by replacing $W(k)$ by $W_n(k)$ in \Cref{crys1} and \Cref{crys2} for each $n \ge 1.$ We denote the corresponding site (equiped with the Zariski topology for simplicity) by $\w{Crys}(k/W_n(k))$ and the associated topos by $(k/W_n(k))_{\w{Crys}}.$ This has a sheaf of rings $\mathcal{O}^{crys}$ defined as before. Analogues to \Cref{why} one has that for  a scheme $X$ over $k$, $R \Gamma _{crys}(X/W_n) \simeq R\text{Hom}_{\w{Crys}(k/W_n(k))}(\mathbb{Z}[\underline{X}], \mathcal O^{crys}),$ where $\underline{X}$ is the sheaf on $\w{Crys}(k/W_n(k))$ defined in a way similar to \Cref{crys3}.
\end{remark}{}

Now we provide a different definition of crystalline cohomology of an algebraic stack following {\cite[2.7.1.]{Ol07}}. Then we will prove that this definition is equivalent to \Cref{def2}.
\begin{definition}\label{def3}Let $\mathscr{X}$ be an algebraic stack over $k.$ We define the \textit{lisse-\'etale crystalline} site of $\mathscr X$ first as the category $\w{Crys}(\mathscr{X})_{ \textit{lis-\'et}}$ whose objects are the data of a $k$-schemes $U$ with a smooth map $U \to \mathscr{X}$ over $k$; a $W(k)$-scheme $T$ on which $p$ is locally nilpotent; a closed $W(k)$-immersion $i: U \to T$ and a divided power structure $\delta$ on the ideal sheaf corresponding to the closed immersion $i$, which we require to be compatible with $(W(k), p, \gamma).$ We will denote an object of $\w{Crys}(\mathscr{X})_{ \textit{lis-\'et}}$ by $(U,T, \delta)$ or simply by $(U,T)$ when no confusion is likely to arise. The morphisms of this category is defined in the obvious way. A family $(U_i, T_i) \to (U,T)$ of maps in $\w{Crys}(\mathscr{X})_{ \textit{lis-\'et}}$ is covering if $U_i = U \times_T T_i$ for all $i$ and $\left \{T_i \to T \right \}$ is an \'etale covering. This equips $\w{Crys}(\mathscr{X})_{ \textit{lis-\'et}}$ with a Grothendieck topology and the resulting site is called the \textit{lisse \'etale crystalline} site. Let $\mathcal{O}^{crys}(U,T):= \Gamma(T, \mathcal O _T)$ be a sheaf of rings on $\w{Crys}(\mathscr{X})_{ \textit{lis-\'et}}$.

\end{definition}

\begin{definition}\label{def4} We define $R
\Gamma_{\textit{lis-\'et-crys}} (\mathscr{X}):= R \Gamma_{}(\w{Crys}(\mathscr{X})_{ \textit{lis-\'et}}, \mathcal{O}^{crys}). $

\end{definition}{}

\begin{proposition}\label{comp}$R\Gamma_{\textit{lis-\'et-crys}} (\mathscr{X}) \simeq R \Gamma_{crys} (\mathscr X),$ i.e., \Cref{def4} is consistent with \Cref{def2}.

\end{proposition}

\begin{proof}We break down the proof in a few steps.\\

\noindent
\textit{Step 1.}
First we fix some notations for the proof. We let $\w{Crys}(\mathscr{X})_{\textit{\'et}}$ denote the big variant of the site in \cref{def3} where we remove the hypothesis that for a pair $(U,T)$ the map $U \to \mathscr{X}$ had to be smooth and the covers are still given by \'etale covers. Sending $(U,T) \to \Gamma(T, \mathcal{O}_T)$ defines a sheaf of rings on both of these sites which we will denote by $\mathcal{O}^{crys}$ in both cases when no confusion is likely to occur. Let $\w{Crys}(\mathscr{X})_{{syn}}$ denote the variant of $\w{Crys}(\mathscr{X})_{\textit{\'et}}$ with the same underlying category but now the covers are given by syntomic covers. Lastly we denote the big syntomic site of $\mathscr{X}$ by $\mathscr X_{syn}$ which is the full subcategory of schemes over $\mathscr X$ and the topology is given by syntomic coverings. \\

\noindent
\textit{Step 2.} In this step, we show that crystalline cohomology can be computed in the big variant $\w{Crys}(\mathscr{X})_{\acute{e}t}.$ We have an inclusion functor $$\w{Crys}(\mathscr{X})_{lis\w{-} \acute{e}t} \to \w{Crys}(\mathscr{X})_{\acute{e}t}$$ which is continuous and cocontinuous and therefore gives a map of topoi $$i:\w{Shv}\w{Crys}(\mathscr{X})_{lis\w{-} \acute{e}t} \to \w{Shv}\w{Crys}(\mathscr{X})_{\acute{e}t}. $$ We want to prove that there is a natural isomorphism $R \Gamma_{lis \w{-} \acute{e}t}(\mathscr{X},i^{-1} \mathcal{O}^{crys}) \simeq R \Gamma_{ \acute{e}t}(\mathscr{X}, \mathcal{O}^{crys}) $. Thus, it is enough to prove that $i^{-1}$ has a left adjoint $i_!$ which is exact. Following the method of proof in {\cite[Tag 07IJ]{SP}}, it is enough to show that $i_! * = *$ where $*$ denotes the final object.\\ 

Let $Y \to \mathcal{X}$ be a smooth cover of $\mathcal{X}$ by a scheme $Y$. First we prove that for any given object $(U,T) \in \w{Crys}(\mathscr{X})_{\acute{et}}$, there exists a cover $(U_i, T_i ) \to (U,T)$ in $\w{Crys}(\mathscr{X})_{ \textit{lis-\'et}}$ such that for all $i$ there exists a map $U_i \to Y$ over $\mathscr{X}.$ In order to show this, first we base change to get a smooth map $U\times_{\mathscr{X}} Y \to U$  which admits a projection map $U\times_{\mathscr{X}} Y \to Y.$ Now we pick an \'etale cover $U' \to U \times_{\mathscr{X}} Y$ where $U'$ is a scheme and thus we get a smooth map $U' \to U.$ This can be refined further to obtain maps $U_i \to U'$ such that the composition $\left \{ U_i \to U' \to U \right \}$ is an \'etale cover.  Now $i: U \to T$ is a divided power thickening of $U$ with $p$ locally nilpotent on $T.$ Therefore $i: U \to T$ is a universal homeomorphism. By invariance of \'etale sites {\cite[Tag 04DY]{SP}} for universal homeomorphisms, we get \'etale maps $T_i \to T$ such that $T_i \times_T U \simeq U_i.$ This gives a thickening $U_i \to T_i$. Since the map $T_i \to T $ is \'etale, by {\cite[Tag 07H1]{SP}}, the ideal of the thickening $U_i \to T_i$ has divided powers. Further, since $\left \{U_i \to U\right \}$ was an \'etale \textit{cover}, $\left \{ T_i \to T  \right \} $ is also an \'etale \textit{cover} by construction. Thus we have constructed a refinement $(U_i , T_i) \to (U,T)$ in the site $\w{Crys}(\mathscr{X})_{ \textit{lis-\'et}}.$ For all $i$, we have also constructed a map $U_i \to Y$ over $\mathscr X.$ 
\\

Now by using the proof in {\cite[Tag 07IJ]{SP}} for the pairs $(U_i, T_i)$ with a map $U_i \to Y$ constructed above, we obtain a set of PD thickenings $(U_s,T_s)_{s \in S} $ with $U_s \to Y$ being an open immersion such that $(U_i, T_i)$ admits a refinement $(U_{ij}, T_{ij})$ in $\w{Crys}(\mathscr{X})_{lis\w{-}\acute{e}t}$ for each $i,j$ and there is a PD map $(U_{ij}, T_{ij}) \to (U_s,T_s)$ in $\w{Crys}(\mathscr{X})_{\acute{e}t}.$ Thus we get an epimorphism $\coprod_{s \in S} h_{(U_s, T_s)} \to *$ in $\w{Shv}\w{Crys}(\mathscr{X})_{\acute{e}t}$ which implies $i_! * = *$ proving our main claim in Step 2.\\


\noindent
\textit{Step 3. }In this step, we show that crystalline cohomology can also be computed in the finer site $\w{Crys}(\mathscr{X})_{syn}$. We have a map of topoi $$ v: \w{Shv}\w{Crys}(\mathscr{X})_{syn} \to \w{Shv}\w{Crys}(\mathscr{X})_{\acute{e}t}.$$ Let $v_*$ denote the right adjoint. The left adjoint $v^{-1}$ is given by sheafification. Thus we obtain a natural isomorphism $R \Gamma_{syn}(\mathscr{X}, \mathcal{O}^{crys}) \simeq  R \Gamma_{\acute{e}t}(\mathscr{X}, R v_* \mathcal{O}^{crys}).$ We claim that $R v_* \mathcal{O}^{crys} \simeq v_* \mathcal{O}^{crys} = \mathcal{O}^{crys}.$ This statement can be checked locally and thus follows from the scheme case where it is already known. We refer to \cite[Prop. 1.1.19.]{BBM} for a proof in the case of schemes.\\

\noindent
\textit{Step 4.} In this step, we show that crystalline cohomology can be computed in $\mathscr{X}_{syn}$ as well and conclude the proof. We have a functor $$h:\w{Crys}(\mathscr X)_{\textit{syn}} \to \mathscr X_{syn}$$ which sends $(V,T) \to V.$ By the proof of \cite[Prop. 1.3.]{FM87}, this functor is cocontinuous. By {\cite[Tag 00XI]{SP}}, there is a morphism of topoi $$u: \w{Shv}\w{Crys}(\mathscr X)_{\textit{syn}} \to \w{Shv} \mathscr X_{syn},$$ where $u_* (\mathcal F) (U):= \varprojlim_{(V,T)/U }\mathcal{F}(V,T). $ We write the left adjoint by $u^{-1}$ which is exact by \cite[Tag 00XL]{SP}.  By adjunction, it follows that $u_*$ commutes with the global section functor: indeed, we obtain $\text{Hom}(*, u_* \mathcal F) \simeq \text{Hom}(*, \mathcal F).$ Therefore it follows that $R \Gamma_{\textit{syn}}(\mathscr X, \mathcal{O}^{crys}) \simeq R \Gamma Ru_* \mathcal{O}^{crys}.$ Now the right hand side is equivalent to \cref{def2}. Therefore we are done. \end{proof}{}



\newpage

\section{Application to crystalline Dieudonn\'e theory}\label{secc3}In this section, we apply the notion of crystalline cohomology of stacks to describe the Dieudonn\'e module of a finite group scheme of $p$-power order and $p$-divisible group as the crystalline cohomology of the classifying stack. Before we do that, we briefly remind ourselves the main theorem of contravariant Dieudonn\'e theory and the related definitions. 
\begin{definition}Let $k$ be a perfect field and $W(k)$ the ring of Witt vectors of $k.$ Let $\sigma$ denote the Witt vector Frobenius which is induced from the Frobenius in $k$. The Dieudonn\'e ring $\mathscr{D}_k$ is defined to be the free non-commutative polynomial ring in two generators $F,V$ over $W(k)$ subjected to the relations $FV= VF = p$, $Fc = \sigma(c)F$ for $c \in W(k)$, and $cV = V \sigma(c) $for $c \in W(k)$. $\mathscr{D}_k$ is non-commutative when $k \neq \mathbb{F}_p$, and is $\mathbb{Z}_p[x, y]/(xy - p)$ when $k = \mathbb{F}_p$.
\end{definition}

\begin{definition}We let $W_n$ denote the group scheme that corepresents the functor that sends a $k$-algebra $A$ to the ring of length $n$ Witt vectors $W_n(A).$ These group schemes are naturally induced with a Frobenius endomorphism $F$ on them and we define $W_n^{m}$ to be the group scheme which is the kernel of $F^m$ on $W_n.$ We also have a map $V: W_n \to W_{n+1}$ induced by the Witt vector Verschiebung which turns $\left \{W_n\right \}_{n\ge 1}$ into a directed system of group schemes. We define $CW^u:= \varinjlim W_n.$ One defines the formal $p$-group $CW$ of Witt covectors as a completion of $CW^u$ in a suitable sense. We refer to \cite{Fo77} for details.

\end{definition}{}

\begin{example}By definition $W_1$ is the additive group scheme $\mathbb{G}_a.$ Also, $W_1^{1}$ is the finite additive group scheme with the underlying scheme given by $\text{Spec}\, k[x]/x^p$ which is often denoted as $\alpha_p.$

\end{example}{}
\begin{definition}For a finite group scheme $G$ over $k$, we define the \textit{Dieudonn\'e module} of $G$ to be $$M(G):=  \w{Hom}(G,CW)$$ where the Hom is being taken in the category of formal groups. 

\end{definition}{}

Now we are ready the state the classical theorem of contravariant Dieudonn\'e theory. 
\begin{theorem}The functor $G \to M(G)$ induces an anti-equivalence between the category of finite group schemes over $k$ and left $\mathscr{D}_k$-modules with finite $W(k)$ length.

\end{theorem}{}

In \cite{BBM}, Berthelot, Breen and Messing obtained an alternative description of $M(G).$ Given a finite group scheme $G$, using \Cref{crys3}, we obtain an abelian group object $\underline{G}$ of the topos $(k/W(k))_{\w{Crys}}.$ They prove the following theorem which expresses $M(G)$ as a certain Ext group in the crystalline topos.
\begin{theorem}[{\cite[Thm. \ 4.2.14.]{BBM}}]\label{gen}For a finite group scheme $G$ for $p$ power rank, we have $\sigma ^* M(G) \simeq \w{Ext}^1_{(k/W(k))_{\w{Crys}}}(\underline{G},\mathcal{O}^{crys} ),$ where $\sigma^* M(G)$ denotes restriction of scalars along the Witt vector Frobenius $\sigma.$

\end{theorem}{}

Lastly, we mention some computations of $M(G)$ for a certain class of group schemes.
\begin{example}Let $\mathscr{D}_n^m := \mathscr{D}_k / (\mathscr{D}_k F^m + \mathscr{D}_k V^n)$. Then we have $M(W_n^m) \simeq \mathscr{D}_n^m . $

\end{example}{}

\subsection{Dieudonn\'e module of finite group schemes}\label{3.1}
In this subsection, we work with finite group schemes $G$ over a perfect field $k$ whose order is a power of $p.$ Let $BG := [ \w{Spec} \,k / G ]$ be the classifying stack of $G$ \cite[Tag 044O]{SP}. Any such group scheme is a local complete intersection and therefore it follows that the map $\w{Spec}\,k \to BG $ is a syntomic cover. Its \v{C}ech nerve is given by the following simplicial scheme: 

\[\xymatrix{
 \cdots  G \times G \times G  \ar[r]<4.5pt>\ar[r]<1.5pt>\ar[r]<-4.5pt>\ar[r]<-1.5pt> &   G \times G  \ar[r]<3pt>\ar[r]\ar[r]<-3pt>  &   G  \ar[r]<1.5pt>\ar[r]<-1.5pt> & \w{Spec}\, k\ .
}\]\\
As in \cref{crys3}, we can attach a sheaf $\underline{G}$ on $\w{Crys}(k/W(k))$ which can be viewed as an object of the topos $(k/W(k))_{\w{Crys}}$. Corresponding to the \v{C}ech nerve above, we obtain by functoriality, a simplicial object of the topos: 

\[\xymatrix{
B\underline{G}:= \cdots \underline{ G} \times \underline{ G} \times \underline{ G}  \ar[r]<4.5pt>\ar[r]<1.5pt>\ar[r]<-4.5pt>\ar[r]<-1.5pt> &  \underline{  G} \times \underline{ G}  \ar[r]<3pt>\ar[r]\ar[r]<-3pt>  &   \underline{ G}  \ar[r]<1.5pt>\ar[r]<-1.5pt> & \ *\ .
}\]\\
where $*$ is the final object. With this simplicial object, we can attach the free simplicial abelian group object:

\[\xymatrix{\cdots \mathbb{Z}[\underline{ G} \times \underline{ G} \times \underline{ G}]  \ar[r]<4.5pt>\ar[r]<1.5pt>\ar[r]<-4.5pt>\ar[r]<-1.5pt> &  \mathbb{Z}[\underline{  G} \times \underline{ G}]  \ar[r]<3pt>\ar[r]\ar[r]<-3pt>  &  \mathbb{Z} [\underline{ G}]  \ar[r]<1.5pt>\ar[r]<-1.5pt> &\ \mathbb{Z}\ .
}\]\\
The alternating face map complex associated to the above simplical object is as an object in the category of complexes of abelian objects of the topos $(k/W(k))_{\w{Crys}}$ which can also be viewed as an object in the derived category of abelian objects of $(k/W(k))_{\w{Crys}}$ denoted as $D(k/W(k))$. This object living in $D(k/W(k))$ is also isomorphic to the homotopy colimit of \[\xymatrix{\cdots \mathbb{Z}^{\w{}}[\underline{  G} \times \underline{ G}]  \ar[r]<3pt>\ar[r]\ar[r]<-3pt>  &  \mathbb{Z}^{\w{}} [\underline{ G}]  \ar[r]<1.5pt>\ar[r]<-1.5pt> &\ \mathbb{Z}\ .
}\] viewed as a functor from $\Delta^{op}$ to $D(k/W(k)).$ We denote this object by $\mathbb Z [B \underline{G}] \in D(k/W(k)).$ If we work with the truncated crystalline site $\w{Crys}(k/W_n(k))$, all of these constructions remain valid and one can associate $\mathbb Z [B \underline{G}] \in D(k/W_n(k))$ as well.

\begin{lemma}\label{hah}$H^0 (\mathbb Z [B \underline{G}]) \simeq \mathbb Z$ and $H^{-1} (\mathbb Z [B \underline{G}]) \simeq \underline{G}.$
\end{lemma}

\begin{proof}
First we work with the presheaf simplicial abelian group object  \[\xymatrix{\cdots \mathbb{Z}^{\w{pre}}[\underline{  G} \times \underline{ G}]  \ar[r]<3pt>\ar[r]\ar[r]<-3pt>  &  \mathbb{Z}^{\w{pre}} [\underline{ G}]  \ar[r]<1.5pt>\ar[r]<-1.5pt> &\ \mathbb{Z}\ .
}\]
 which is obtained by applying the free abelian presheaf functor to $B \underline{G}$ \cite[Tag 03CP]{SP}. Let $\mathbb{Z}^{\w{pre}}[B\underline{G}]$ be the associated complex: $K^{\bullet}:=\cdots \mathbb{Z}^{pre}[\underline{G}\times \underline{G}] \to \mathbb{Z}^{pre}[\underline{G}] \to \mathbb{Z} \to 0$. We see that $H^0$ of this complex is $\mathbb{Z}$ since the last differential is zero.  One also notes that $K^\bullet$ computes the (presheaf) group cohomology of the (presheaf) abelian group $\underline{G}$ with constant coefficients in $\mathbb Z$ and hence, $H^{-1}(K^\bullet) \simeq \underline{G}.$ Now since sheafification is exact, we obtain the required statements. \end{proof}
   
     

\begin{lemma}\label{lol1}Let $G$ be a group scheme of order $p^m.$ Then $H^{-i}(\mathbb Z [B \underline{G}])$ is killed by $p^m$ for $i>0.$

\end{lemma}{}

\begin{proof}We start by recalling some definitions for this proof. Let $A$ be an ordinary $n$ torsion abelian group.  Let us define $\mathbb{Z} [BA]$ to be the alternating face complex associated to the simplicial abelian group \[\xymatrix{\cdots  \mathbb{Z}[A \times A]  \ar[r]<3pt>\ar[r]\ar[r]<-3pt>  &  \mathbb{Z} [A]  \ar[r]<1.5pt>\ar[r]<-1.5pt> &\ \mathbb{Z}\ .
}\]
 Then as noted previously, $H^{-i}(\mathbb Z [B A]) \simeq H_i ( A)$, where the latter denotes group cohomology with coefficients in $\mathbb Z.$ This is also the homology of the Eilenberg-MacLane space $K(A,1).$ We will show that $H_i(A)$ is $n$-torsion (*). First, from the complex $\mathbb{Z}[BA]$, one notes that $H_i(A)$ commutes with filtered colimits as a functor defined on abelian groups. Since $A$ is an $n$-torsion abelian group, it can be expressed as filtered colimit of finite $n$-torsion abelian groups. Therefore, it is enough to check the statement for an $n$-torsion finite abelian group $A$. Further, one notes that $K(A_1,1) \times K(A_2,1) \simeq K(A_1 \times A_2, 1),$ therefore, by using the K\"unneth formula, we are reduced to checking this for $n$-torsion cyclic groups, where it follows from the well-known computation of group homology of finite cyclic groups. The statement in the lemma now follows from applying (*) to $\mathbb{Z}^{\w{pre}}[B \underline{G}].$ \end{proof}{}
   
\begin{remark}\label{rem}
These lemmas clearly remain valid even if we were working with the truncated crystalline sites $\w{Crys}(k/W_n(k))$ mentioned in \Cref{trun}.
\end{remark}{}

\begin{proposition}\label{lol2}
$R \Gamma _{crys}(BG) \simeq  R\w{Hom}_{D(k/W(k))}(\mathbb{Z}[B \underline{G}], \mathcal O^{crys}).$
\end{proposition}{}
\begin{proof}
Since $\text{Spec}\, k \to B G$ is a syntomic map, we can apply \cref{desc} which gives that 

\[ R \Gamma_{crys}(BG) \simeq R \varprojlim \big(\cosimp {R \Gamma_{crys}(\w{Spec}\, k)}{R \Gamma_{crys}(G)}{R \Gamma_{crys}(G \times G)} \big).\] Which by \Cref{why} is  
      \[ \simeq R \varprojlim \big(\cosimp {R\w{Hom}_{D(k/W(k))}(\mathbb Z, \mathcal O^{crys})}{R\w{Hom}_{D(k/W(k))}(\mathbb Z[\underline{G} ],\mathcal O^{crys} )}{R\w{Hom}_{D(k/W(k))}(\mathbb Z [\underline{G}\times \underline{G}], \mathcal O^{crys})} \big) .\] We can take the $R \varprojlim$ inside as a homotopy colimit, which gives us that the above is $\simeq R \w{Hom}_{D(k/W(k))}(\mathbb Z [B\underline{G}], \mathcal O^{crys}),$ as desired. \end{proof}{}

\begin{proposition}\label{lol3}
There is a spectral sequence with $E_2$-page $$E_2^{i,j}= \w{Ext}_{(k/W(k))_{\w{Crys}}}^i (H^{-j}(\mathbb{Z}[B \underline{G}], \mathcal{O}^{crys})\implies H^{i+j}_{crys}(BG), $$ and another spectral sequence with $E_1$-page 
$$ E_1^{i,j}=H^j_{crys}(G^i) \implies H^{i+j}_{crys}(BG), $$ where $G^i$ denotes the $i$ fold fibre product of $G$ with itself. By convention, $G^0 = *.$

\end{proposition}{}

\begin{proof}
This is a consequence of \cite[Tag 07A9]{SP}, \Cref{lol2} and \Cref{why}.
\end{proof}{}

\begin{proposition}\label{lol4}
$H^1_{{crys}}(BG) = 0.$
\end{proposition}

\begin{proof}
We can use the $E_2$ spectral sequence from \Cref{lol3}. to compute $H^i_{crys} (BG)$. We note that $\w{Ext}^i (\mathbb Z, \mathcal O^{crys})=0$ for $i >0$ as it computes the cohomology of $\w{Spec}\, k$ for a perfect field $k$ by \Cref{why}. Also, by \cite[Prop. 4.2.6.]{BBM}, we have $\w{Hom}(\underline{G}, \mathcal O^{crys})= 0.$ These calculations along with the spectral sequence and \Cref{hah} imply $H^1_{{crys}}(BG) = 0$.
\end{proof}{}

\begin{proposition}\label{lol5}
If $G$ is a finite group scheme of $p$-power order, then for any $i>0$, $H^i_{crys}(BG)$ is killed by a power of $p$ as an abelian group.
\end{proposition}{}

\begin{proof}This follows from the spectral sequence in \Cref{lol3}. Indeed, already in the $E_2$-page of the spectral sequence, all but $E_2^{0,0}$ is $p$-power torsion by \Cref{lol1}. Hence, all but $E^{0,0}_{\infty}$ is $p$-power torsion as well. Therefore, for any $i>0$, $H^i_{crys}(BG)$ has a finite filtration whose successive quotients are $p$-power torsion and hence is $p$-power torsion itself. \end{proof}{}

\begin{proposition}\label{lol6}If $G$ is a finite group scheme of $p$-power order, then $H^2_{crys}(BG) \simeq \w{Ext}^1_{(k/W(k))_{\w{Crys}}}(\underline{G}, \mathcal{O}^{crys}).$

\end{proposition}

\begin{proof}
First we note that there is a natural map $\w{Ext}^1_{(k/W(k))_{\w{Crys}}}(\underline{G}, \mathcal{O}^{crys}) \to H^2_{crys}(BG)$, which is injective. Indeed, from the $E_2$-spectral sequence (which is natural in $G$), we note that $\w{Fil}^0{H^2_{crys}(BG)} = H^2_{crys}(BG),$ $\w{Fil}^i{H^2_{crys}(BG)}= 0$ for $i \ge 2$ and therefore, $\w{Fil}^1{H^2_{crys}(BG)} = \w{Ext}^1_{(k/W(k))_{\w{Crys}}}(\underline{G}, \mathcal{O}^{crys})$, which gives the required injective map. We proceed to proving that this natural map is an isomorphism. \\

Let $R\Gamma_{crys}(BG/W_n):= R\Gamma_{crys}(BG) \otimes^{L}_{W(k)} W_n(k)$ and $H^i_{crys}(BG/W_n):= H^i (R\Gamma_{crys}(BG/W_n)).$ Then we have the following exact sequence: $$0 \to H^1_{crys}(BG)/p^n \to H^1_{crys}(BG/W_n) \to H^2_{crys}(BG)[p^n]\to 0. $$
Now we choose an $n$-large enough such that $p^n G = 0$ and $H^2_{crys}(BG)$ is $p^n$-torsion. This is possible by \Cref{lol5}. Then the above exact sequence along with \Cref{lol4} gives $H^1_{crys}(BG/W_n) \simeq H^2 _{crys}(BG)$ for such an $n.$ By \cite[Prop. 4.2.17.]{BBM}, we have $\w{Ext}^1_{(k/W(k))_{\w{Crys}}}(\underline{G}, \mathcal{O}^{crys}) \simeq \w{Hom}_{(k/W_n(k))_{\w{Crys}}}  (\underline{G}, \mathcal O^{crys}).$ Therefore, it is enough to show that $H^1_{crys}(BG/W_n) \simeq \w{Hom}_{(k/W_n(k))_{\w{Crys}}}  (\underline{G}, \mathcal O^{crys}).$\\

Since \[ R \Gamma_{crys}(BG) \simeq R \varprojlim \big(\cosimp {R \Gamma_{crys}(\w{Spec}\, k)}{R \Gamma_{crys}(G)}{R \Gamma_{crys}(G \times G)} \big),\]\\
we obtain (by switching limit)
\[ R \Gamma_{crys}(BG/W_n) \simeq R \varprojlim \big(\cosimp {R \Gamma_{crys}(\w{Spec}\, k/W_n)}{R \Gamma_{crys}(G/W_n)}{R \Gamma_{crys}(G \times G/W_n)} \big).\]\\
 Therefore, by \Cref{rem} and the proof of \Cref{lol2}, one obtains that 
$$R \Gamma _{crys}(BG/W_n) \simeq  R\w{Hom}_{D(k/W_n(k))}(\mathbb{Z}[B \underline{G}], \mathcal O^{crys}),$$
and analogous to \Cref{lol3} a spectral sequence with $E_2$-page $$E_2^{i,j}= \w{Ext}_{(k/W_n(k))_{\w{Crys}}}^i (H^{-j}(\mathbb{Z}[B \underline{G}], \mathcal{O}^{crys})\implies H^{i+j}_{crys}(BG/W_n). $$
Now applying \Cref{hah} yields the desired isomorphism. To check this, we see that $E^{1,0}_2 = 0$ and $E^{0,1}_2 = \w{Hom}_{(k/W_n(k))_{\w{Crys}}}(\underline{G}, \mathcal{O}^{crys}).$ Also we have $E_2^{2,0}= 0$. Since $E^{i,j}_2 = 0$ for $i<0$ or $j<0$, we obtain that $E^{0,1}_{\infty}= \w{Hom}_{(k/W_n(k))_{\w{Crys}}}(\underline{G}, \mathcal{O}^{crys}).$ Therefore, we indeed obtain the required isomorphism $H^1_{crys}(BG/W_n) \simeq \w{Hom}_{(k/W_n(k))_{\w{Crys}}}  (\underline{G}, \mathcal O^{crys}).$ \end{proof}{}

\begin{proof}[Proof of \cref{2}]\, Follows from \Cref{lol6} and \cref{gen}.
\end{proof}

\begin{remark}\label{higher} One can develop the theory of crystalline cohomology for higher stacks and study the cohomology of the $n$-stack $K(G,n)$ \cite{To06}. For a group scheme $G$, the $n$-stack $K(G,n)$ is supposed to be analogue of the Eilenberg-MacLane space $K(G,n)$ for any discrete abelian group $G$ which has the property that $\pi_i(K(G,n))= G$ for $i=n$ and $\pi_i(K(G,n))= 0$ for $i\ne \left \{n,0\right  \} $. In the topological case, there exists a chain complex of abelian groups $\mathbb{Z}[B^n G]$ such that $H_i(\mathbb{Z}[B^n G])$ computes the singular homology of the CW complex $K(G,n).$ Here, we do not define crystalline cohomology for higher stacks in general, however, for a finite group scheme $G$, we work with an ad hoc definition generalizing \Cref{lol2}. Similar to the definition of $\mathbb{Z}[B \underline{G}]$ as an object of the crystalline topos, one can also define $\mathbb{Z}[B^n \underline{G}]$ as an object of the crystalline topos. Then we define  $R\Gamma_{crys}(K(G,n)): = R\w{Hom}_{(k/W(k))}(\mathbb{Z}[B^n G], \mathcal{O}^{crys}).$ Then from the analogue of the spectral sequence in \Cref{lol3}, we obtain that $H^i_{crys}(K(G,n))=0$ for $0<i<n+1$ and $H^{n+1}_{crys}(K(G,n))= M(G).$ In order to prove this for $n \ge 2$, the computation with spectral sequence relies on the fact that for an abelian group $A$, $H_i(K(A,n), \mathbb{Z})= 0$ for $0<i<n$, $H_n(K(A,n), \mathbb{Z})= A$ and $H_{n+1}(K(A,n), \mathbb{Z})=0.$ First two of these computations follow from Hurewicz theorem and the last one follows from applying the Serre fibration spectral sequence for the homotopy fibration sequence $K(A,n) \to * \to K(A, n+1).$ 

\end{remark}

\subsection{Dieudonn\'e theory of $p$-divisible groups}\label{keno} At first, we recall the definition of a $p$-divisible group.
\begin{definition} Let $p$ be a prime and $h>0$ an integer. A $p$-divisible group (or Barsotti Tate
group) of height $h$ over a scheme $S$ is a directed system $G = \left \{ G_n\right \}$ of finite flat
group schemes over $S$ such that each $G_n$ is $p^n$
-torsion of order $p^{nh}$ and the transition map
$i_n : G_n \to G_{n+1}$ is an isomorphism of $G_n$ onto $G_{n+1}[p^n]$ for all $n \ge 1$.
A morphism $f : G \to H $ between $p$-divisible groups is a compatible system of $S$-group
maps $f_n : G_n \to H_n$ for all $n \ge 1$.
If $S' \to S$ is a map of schemes then $G 
\times _S S'
:= {G_n \times _S S'}$ is the $p$-divisible group of
height $h$ over $S'$ obtained by base change.

\end{definition}

\begin{example} If $A \to S$ is an abelian scheme with fibers of constant dimension $g > 0$ and $p$ is a prime, then $G_n := A[p^n]$ is a finite flat $p^n$-torsion group scheme of order $p^{2gn}$ for all $n\ge1$ and $\left \{G_n\right \}$ is a directed system via isomorphisms $G_n \cong G_{n+1}[p^n]$ for all $n\ge1$. For an abelian variety $A$, we denote this $p$-divisible group by $A[p^\infty]$. It has height $2g$ where $g = \dim A$.
\end{example}

We also recall the following theorem on Dieudonn\'e theory of $p$-divisible groups.

\begin{theorem}\label{p}The functor $G \to M(G) := \varprojlim M(G_n)$ is an anti-equivalence of categories between the category of $p$-divisible groups over $k$ and the category of finite free $W(k)$- modules $D$ equipped with a Frobenius semilinear endomorphism $F:D \to D$ such that $pD \subseteq F (D)$. The height of $G$ is the $W(k)$-rank of $M(G).$ 

\end{theorem}

Now in order to formulate the analogue of \Cref{2} for $p$-divisible groups, we begin by making a definition.

\begin{definition}
Let $G$ be a $p$-divisible group over $k.$ We define the classifying stack of $G$ to be the  $B G :=  \varinjlim B G_n $ where the colimit is being taken in the category of stacks.
\end{definition}{}

\begin{proposition}\label{p'}For a $p$-divisible group $G= \left \{G_n \right \}$ we have 
$R\Gamma_{crys}(BG)\simeq R \varprojlim R \Gamma_{crys} B G_n. $
\end{proposition}{}

\begin{proof}
We write $\mathcal{F}_{n}:= BG_n$ and $\mathcal{F}:= BG$ for this proof. Using \Cref{aff}, we see that crystalline cohomology of a stack $\mathcal{Y}$ only depends on $\mathcal{Y}$ viewed as a presheaf of groupoids on the category of affine $k$-schemes $\w{Aff}_k.$ If $\mathcal{Y}:= \w{colim} \mathcal{Y}_\alpha$ as presheaves of groupoids, then by the alternative definition in \cref{aff} it follows that $R \Gamma _{crys}(\mathcal{Y}) \simeq R\,\w{lim}_{\alpha} R \Gamma_{crys}(\mathcal{Y}_{\alpha}) .$ Therefore, in order to show that $R \Gamma _{crys}(\mathcal{F}) \simeq R\,\varprojlim_{n} R \Gamma_{crys}(\mathcal{F}_{n}),$ it is enough to prove that $\mathcal{F}$ is the colimit of $\mathcal{F}_n$ in the category of presheaves of groupoids on $\w{Aff}_k.$ This follows since affine schemes in the fpqc site of all schemes satisfy the property that any fpqc covering of an affine scheme has a finite subcovering by affine schemes. Indeed, by the property we mentioned, checking descent over an affine scheme is essentially a finite limit condition. Thus our claim follows since filtered colimit commutes with finite limits. 
\end{proof}{}

\begin{proposition}\label{apt}
We have a natural isomorphism $H^2_{crys}(BG) \simeq M(G).$
\end{proposition}{}

\begin{proof} We are interested to compute $H^2_{crys}(BG)$ which is $H^2 (R \varprojlim R \Gamma_{crys} B G_n)$ by \Cref{p'}. We compute the cohomology of this $\mathbb N$-indexed derived limit by using \cite[Tag 07KY]{SP}, which gives us the following short exact sequence 

$$ 0 \to R^1 \varprojlim H^1_{crys} (B G_n) \to H^2_{crys}(BG) \to \varprojlim H^2_{crys}(BG_n) \to 0. $$
Therefore our claim follows from \Cref{lol4}, \Cref{lol6}, and \Cref{p}. \end{proof}{}

\subsection{Cohomology of the classifying stack of abelian varieties}\label{ab}Now we consider an abelian variety $A$ over the field $k.$ Let $BA$ denote the classifying stack of $A.$ In this section, in \Cref{cw}, we prove that $H^{2*}_{crys}(BA) \cong \text{Sym}^* (H^1_{crys}(A)), $ and $H^i_{crys}(BA) = 0$ for odd $i$ .

\begin{remark}An analogue of this proposition was proven by Borel in topology \cite{Bo53}, and by Behrend in $\ell$-adic cohomology using fibration spectral sequences {\cite[Thm. \ 6.1.6.]{Beh03}}. Here we take a different approach based on descent theory. One knows that there is a functorial isomorphism $H^*_{crys}(A) \cong \bigwedge^*H^1_{crys}(A).$ We show that \Cref{cw} is a formal consequence of this isomorphism and the K\"unneth formula.  

\end{remark}{}

In this section, we will crucially use the theory of derived functors of non-additive functors which appears in \cite{DP61}, \cite{Il71}. Note that by \Cref{lol3}, there is a $E_1$ spectral sequence with $E_1^{i,j}= H^j_{crys}(A^i) \implies H^{i+j}_{crys}(BA),$ where $A^i$ denotes the $i$-fold fibre product of $A$ with itself. We note that by definition, for $j \ge 1$, $E_1^{\bullet,j }$ is the alternating face complex associated to the cosimplicial object given by

\[  \big(\cosimp {H^j_{crys}(*)}{H^j_{crys}(A)}{ H^j_{crys}(A \times A )} \big);\] which is naturally isomorphic to the cosimplicial object \[\bigwedge^j  \big(\cosimp {H^1_{crys}(*)}{H^1_{crys}(A)}{ H^1_{crys}(A \times A )} \big).\]

\begin{lemma}\label{hom}
The complex $E_1^{\bullet, 1}$ is homotopy equivalent to $H^1_{crys}(A) [-1].$
\end{lemma}{}
\begin{proof}As noted above, $E_1^{\bullet, 1}$ is the alternating face complex associated to \[ \big(\cosimp {H^1_{crys}(*)}{H^1_{crys}(A)}{ H^1_{crys}(A \times A )} \big).\] By Kunneth formula, one has a natural isomorphism $H^{1}_{crys} (A^i) \cong \bigoplus_{s=1}^{i} H^1_{crys}(A),$ and $H^1(*) = 0.$ Writing $H^1_{crys}(A)=V$, we note that $E_1^{\bullet,1}$ is naturally isomoprphic to the following explicit complex (where $V$ sits in cohomological degree 1)
   
   $$K_1^{\bullet}:\,\, \ldots \to 0 \to V \to V^{\oplus 2}\to V^{\oplus 3} \to \ldots . $$
Here the first differential $d_1: V \to V^{\oplus 2}$ is zero. In general, for $n \ge 2$, the differential $d_n: V^{\oplus n} \to V^{\oplus n+1}$ is given by 

$$d_n(v_1, v_2, \ldots, v_n)= (-v_1, 0, v_2 - v_3, 0, v_4 - v_5, 0, \ldots, v_n)\, \, \w{for even } n, $$
 and 
 
 $$d_n(v_1, v_2, \ldots, v_n) = (0, v_2, v_2, v_4, v_4, \ldots, 0)\,\, \text{for odd } n. $$\\
One can check the above by induction. Now we consider the complex $K_2^\bullet: \, \, \ldots \to 0 \to V \to 0 \to 0 \ldots$ where $V$ sits in cohomological degree $1.$ There are obvious maps from $K_1^\bullet \to K_2^\bullet$ and $K_2^\bullet \to K_1^\bullet$ (induced by $\w{id}_V$ on cohomological degree $1$) and we prove that they induce the desired homotopy equivalence. There is nothing to prove for the composite map from $K_2^\bullet \to K_2^\bullet.$ For the other composite map $K_1^\bullet \to K_1^\bullet$, our task is to prove that it is homotopic to $\w{id}_{K_1^{\bullet}}$. We construct the required homotopy $h_i$. We set $h_i = 0$ for $i \le 2.$ For $i \ge 3,$ now we define $h_i: V^{\oplus i} \to V^{\oplus {i-1}}.$ We let $h_3(v_1,v_2,v_3) = (-v_1, v_3)$ and $h_4(v_1, \ldots, v_4) = (0,v_2, 0).$ In general, we set 
$$h_i(v_1, \ldots, v_i)= (0, v_2,0, \ldots, 0,v_{i-2}, 0)\, \w{for even } i \ge 6, $$
and
$$ h_i(v_1, \ldots, v_i)= (-v_1, 0 , -v_3,0, \ldots,0,-v_{i-2}, v_i ) \, \w{for odd } i \ge 5.$$
One easily checks that this gives the required homotopy. \end{proof}{}

\begin{lemma}\label{hom1}
The complex $E_1^{\bullet, n}$ is homotopy equivalent to $\text{Sym}^{n}(H^1_{crys}(A))[-n].$
\end{lemma}
\begin{proof}
Writing $H^1_{crys}(A)= V$ again, we note that by \Cref{hom}, $E_1^{\bullet, 1}$ is homotopy equivalent to $V[-1].$ Therefore, the (derived) $n$-th exterior power $\wedge^n (V [-1]) $ is homotopy equivalent to $\wedge^n (E_1^{\bullet, 1})$ by \cite[Def. 4.1.3.2]{Il71} and Lemma  4.1.3.5. from loc. cit. It also follows from Definition 4.1.3.2 and Section 1.3.4 in loc. cit. that $\wedge^n (E_1^{\bullet, 1})$ is homotopy equivalent to the alternating face complex associated to \[\bigwedge^n  \big(\cosimp {H^1_{crys}(*)}{H^1_{crys}(A)}{ H^1_{crys}(A \times A )} \big),\] which is isomorphic to $E_1^{\bullet,n}.$ Now one also notes that $\wedge^n (V[-1]) $ is homotopy equivalent to $\text{Sym}^n (V)[-n],$ by the formula of d\'ecalage (i) in Proposition 4.3.2.1. loc. cit. This proves the lemma.
\end{proof}{}

\begin{proposition}\label{cw}We have a natural isomorphism $H^{2*}_{crys}(BA) \cong \text{Sym}^* (H^1_{crys}(A)), $ and $H^i_{crys}(BA) = 0$ for odd $i$.

\end{proposition}{}

\begin{proof}
This follows from the spectral sequence, \Cref{hom} and \Cref{hom1}. The spectral sequence also guarantees the naturality of the isomorphisms.  
\end{proof}{}

\begin{corollary}For an abelian variety $A$ over $k$, $H^2_{crys}(BA)$ is naturally isomorphic to the Dieudonn\'e module associated to the $p$-divisible group $A[p^\infty].$ 

\end{corollary}{}

\begin{proof}
This follows from \Cref{cw} and the fact that $H^1_{crys}(A)$ is isomorphic to the Dieudonn\'e module associated to the $p$-divisible group $A[p^\infty].$
\end{proof}{}

\subsection{Cohomology of the classifying stack of $p$-divisible groups}\label{man}
Let $G$ be a $p$-disivisible group over a perfect field $k.$ Using the calculation in \cref{ab}, we are able to fully compute the cohomology ring $H^* _{crys}(BG).$ In this section, we prove that $H^*_{crys}(BG) \simeq \w{Sym}^{2*} (M(G)),$ where $M(G)$ is the Dieudonn\'e module of $G.$ Our strategy is to compare the $p$-divisible group with an abelian variety and deduce the result from \cref{cw}. First, we will record two lemmas.

\begin{lemma}\label{camera}Let $G$ be a uniquely $p$-divisible abelian group. Then $H_i (G, k)=0$ for $i>0$ where $H_i(G,k)$ denotes the group homology with coefficients in the field $k$ equipped with trivial $G$ action.

\end{lemma}{}
\begin{proof}
Since $G$ is an abelian group there is an exact sequence $ 0 \to \mathbb{Z}^{\oplus I} \to \mathbb{Z}^{\oplus J} \to G \to 0.$ By taking colimits over multiplication by $p$, we obtain an exact sequence $0 \to \mathbb{Z}[1/p]^{\oplus I} \to \mathbb{Z}[1/p]^{\oplus J} \to G \to 0.$ Using the Hochschild Serre spectral sequence, we are reduced to checking the claim for $G = \mathbb Z[1/p]^{\oplus I}.$ By taking filtered colimits, we can assume that $I$ is finite. Then by Kunneth formula it is enough to check the claim for $\mathbb{Z}[1/p].$ By taking filtered colimits over multiplication by $p$, it follow from the fact that $H_1 (\mathbb{Z}, k)= k$ and $H_i (\mathbb{Z}, k)= 0$ for $i>0.$
\end{proof}{}

\begin{lemma}\label{tv}
Let $G$ be an abelian group such that multiplication by $p$ is surjective on $G.$ Then it follows that $W_n(k)[B G[p^\infty]] \to W_n(k)[BG]$ is an isomorphism in the derived category $D(W_n(k))$ of $W_n(k)$-modules. Here $G[p^\infty] := \varinjlim G[p^n].$
\end{lemma}{}

\begin{proof}
By going modulo $p$, it would be enough to show that $k [B G[p^\infty]] \to k[BG]$ is quasi-isomorphism. For that, it would be enough to prove that map induced on group homology $H_i (G[p^\infty], k) \to H_i (G, k)$ is an isomorphism. Here, in both cases group cohomology is taken with constant coefficients in $k.$ We have the following exact sequence 

$$0 \to G[p^\infty] \to G \to \varinjlim_{p} G \to 0.$$We will write $Q:=\varinjlim_{p} G.$ By the Hochschild Serre spectral sequence we have a spectral sequence $$E_2 = H_p (Q, H_q(G[p^\infty],k)) \implies H_{p+q}(G,k).$$ Since $G$ is abelian and $G$ acts trivially on $k$, it follows that $Q$ acts trivially on $H_q(G[p^\infty], k).$ Since $k$ is a field $H_q(G[p^\infty], k)$ is a direct sum of copies of $k$ and therefore, we are done by \cref{camera}.
\end{proof}{}

\begin{proposition}\label{person}Let $A$ be an abelian variety over $k.$ Let $A[p^\infty]$ be the associated $p$-divisible group. Then $H^*_{crys}(B(A[p^\infty])) \simeq \text{Sym}^{2*}(H^2 _{crys}(BA[p^\infty]))$.

\end{proposition}{}

\begin{proof}
We let $\mathcal{X}$ denote the topos $\text{Shv(SYNSch}_k)$. Given the abelian variety $A$, we can associate an object in $\mathcal{X}$ (by considering functor of points) which will also be denoted by $A.$ Similarly, by taking direct limits, we can associate an object in $\mathcal{X}$ corresponding to $A[p^\infty]$ which will again be denoted by $A[p^\infty].$ Since they are both group objects, we can consider the associated classifying objects as a simplicial object in $\mathcal{X}$ denoted respectively by $B A$ and $B A[p^\infty].$ One can also consider the free $W_n(k)$-module on these objects and take the associated alternating face complex which will be denoted by $W_n(k) [B A]$ and $W_n(k) [B A[p^\infty]]$ respectively which can both be viewed as objects in $D(W_n(k))_{\mathcal{X}}$; the derived category of $W_n(k)$-modules in $\mathcal{X}.$. We let $R \Gamma_{crys}(B A[p^\infty]/W_n):= R\text{Hom}_{D(W_n(k))_{\mathcal{X}}}(W_n(k)[B A[p^\infty]], \mathcal{O}^{crys}/W_n)$ and $R \Gamma_{crys}(BA/W_n) := R\text{Hom}_{D(W_n(k))_{\mathcal{X}}}(W_n(k)[B A], \mathcal{O}^{crys}/W_n).$

\begin{lemma}\label{woman}There is a natural isomorphism $W_n(k)[B A[p^\infty]] \simeq W_n(k) [B A]$ in $D(W_n(k))_{\mathcal{X}}.$ Thus $$R \Gamma _{crys} (B A[p^\infty]/W_n) \simeq R \Gamma_{crys} (BA/W_n).$$

\end{lemma}{}

\begin{proof}
Since multiplication by $p$ map on the abelian variety $A$ is a syntomic cover, it follows that as an abelian group object of $\mathcal{X},$ multiplication by $p$ is surjective on $A.$ Using the map $A[p^\infty] \to A$, we obtain a natural map $W_n(k)[BA[p^\infty]] \to W_n(k)[BA]$ and in order to prove that this map is a quasi-isomorphism, we need to check that $H^i (W_n(k)[BA[p^\infty]]) \to H^i (W_n(k)[BA])$ is isomorphism. Since the topos has enough points, it is enough to check this by taking stalks at a (geometric) point $x: \text{Sets} \to \mathcal{X}$. But since taking stalk is an exact functor, it commutes with taking cohomology. Therefore, we can take stalks levelwise on the complexes associated to $W_n(k)[B A[p^\infty]]$ and $W_n(k)[BA]$ and then check that it is a quasi-isomorphism. Again, by noting that taking stalks commutes with the free $W_n(k)$-module object construction, we are reduced to proving that there is a natural quasi-isomorphism $W_n(k)[BA_x[p^\infty]] \to W_n(k)[B A_x]$ of complexes of $W_n(k)$-modules. This follows from \cref{tv}.
\end{proof}{}
By taking inverse limits over $n$, we obtain $R \Gamma_{crys}(B A[p^\infty]) \simeq R \Gamma_{crys}(BA).$ Now \cref{person} follows from the fact that $H^*_{crys}(BA)$ is a symmetric algebra in $H^2_{crys}(BA)$ by \cref{cw}. \end{proof}{}

\begin{proposition}\label{hislast}Let $G$ be any $p$-divisible group over $k$. Then $H^*_{crys} (BG)= \text{Sym}^{2*} (H^2_{crys}(BG)).$
\end{proposition}{}

\begin{proof}
By \cite{Oor01}, there exists a $p$-divisible group $G'$ such that $G \times G' \simeq A[p^\infty]$ for some abelian variety $A.$ Using the zero section of $G'$ we have a map $G \to G \times G'$ which when composed with projection to $G$ gives the identity map. Thus $BG$ is a retract of $BA[p^\infty].$ Since $H^*_{crys}(BA[p^\infty])$ is generated as a (commutative) symmetric algebra in degree $2$ classes, it follows that $H^*_{crys}BG$ is generated by degree $2$ classes, i.e., there is a surjection $\text{Sym}^{2*} (H^2 _{crys}(BG)) \to H^*_{crys}(BG).$ Doing the same for $G'$ we arrive at the following commutative diagram.

\begin{center}
\begin{tikzcd}
\text{Sym}^{2^*}H^2_{crys}(BG)\otimes \text{Sym}^{2^*}H^2_{crys}(BG') \arrow[d, two heads] & {} \arrow[rr, "\simeq"] &  & {} & \text{Sym}^{2*}(H^2_{crys}(BG \times BG')) \arrow[d, "\simeq"] \\
H^*_{crys}(BG)\otimes H^*_{crys}(BG')                                                      & {} \arrow[rr]           &  & {} & H^*_{crys}(BG \times BG')                              
\end{tikzcd}
\end{center}{}
The upper horizontal map is an isomorphism since $H^2_{crys}(BG)$ is isomorphic to the Dieudonn\'e module (\cref{apt}) and therefore, $H^2_{crys}(BG \times BG') \simeq H^2_{crys}(BG) \oplus H^2_{crys}(BG').$ By using \cref{person}, the right vertical map is an isomorphism since $BG \times BG' \simeq B A[p^\infty]$. This shows that the surjection $\text{Sym}^{2*} (H^2 _{crys}(BG)) \to H^*_{crys}(BG)$ must be an isomorphism. \end{proof}{}

\begin{proof}[Proof of \cref{alvida}]Follows from \cref{hislast} and \cref{apt}.
\end{proof}

\newpage

\section{Prismatic cohomology}\label{sec4}
\subsection{Prismatic cohomology for stacks} In this section, we start by recalling the notion of prismatic cohomology and then extend the notion of prismatic cohomology to stacks. The main reference for this section is \cite{BS19} and \cite{BMS19}. We will freely use the definitions and notations from their paper. Our somewhat terse exposition here is also loosely based on \cite{AL19}.

\begin{definition}{\cite[Def. 4.10.]{BMS19}} A ring $S$ is called \textit{quasi-syntomic} if $S$ is $p$-complete with bounded $p^\infty$-torsion and the $p$-adically completed cotangent complex $\widehat{\mathbb{L}_{S/\mathbb{Z}_p}}$ has Tor-amplitude in $[-1,0].$ The category of all quasi-syntomic rings is denoted by $\mathbf{Qsyn}.$ A map $S \to S'$ of $p$-complete rings with bounded $p^\infty$-torsion is a \textit{quasi-syntomic morphism} if $S'$ is $p$-completely flat over $S$ and $\widehat{\mathbb{L}_{S'/S}}$ has Tor-amplitutde in $[-1,0].$ A quasi-syntomic morphism is called a \textit{quasi-syntomic cover} if the map $S \to S'$ is $p$-completely faithfully flat i.e., $S/p^n \to S'/p^n$ is faithfully flat for all $n \ge 1.$

\end{definition}{}

\begin{definition}{\cite[Def. 4.18.]{BMS19}} A ring $S$ is called (integral) \textit{perfectoid} if $S$ is $p$-complete, such that $\pi^p = p u$ for some $\pi \in S$, $u \in S ^\times$; the Frobenius is surjective on $S/p$ and kernel of the map $\theta: A_{\w{inf}}(S) := W(S^\flat) \to S$ is principal.

\end{definition}{}

\begin{definition}{\cite[Def. 4.20.]{BMS19}} A ring $S$ is called \w{quasi-regular semiperfectoid}(QRSP) if $S \in \mathbf{Qsyn}$ and there exists a perfectoid ring $R$ mapping surjectively onto $R.$

\end{definition}{}

\begin{definition} If $R$ is any $p$-complete ring we will let $(R)_{\mathbf{Qsyn}}$ denote the (opposite) category of all rings over $R$ which are quasi-syntomic. This category can be equipped with a Grothendieck topology generated by quasi-syntomic covers which makes this a site.

\end{definition}{}

\begin{remark}\label{qr} By the results in \cite[Sec. 4.4.]{BMS19}, QRSP rings form a basis for the site $(R)_{\mathbf{Qsyn}}.$ Therefore, in order to specify a sheaf on $(R)_{\mathbf{Qsyn}}$ amounts to assigning values to the QRSP rings satisfying the descent condition.

\end{remark}

In \cite{BS19} the authors define prismatic cohomology at first for $p$-complete smooth algebras and then extend it to all $p$-complete algebras \cite[Construction 7.6.]{BS19} and call it the \textit{derived prismatic cohomology}. An important fact \cite[Prop. 7.10.]{BS19} about this construction is that for every QRSP ring $S$, the derived prismatic cohomology denoted as $\Prism_S$ lives only in degree 0, i.e., is discrete. Also, the functor $S \to \Prism_S$ forms a quasi-syntomic sheaf \cite[Construction 7.6. (3)]{BS19}. Using these facts the authors in \cite{AL19} construct a sheaf of rings $\mathcal{O}^{\w{pris}}$ on $(R)_{\mathbf{Qsyn}}$ such that for any QRSP ring $S$ over $R$, one has $$\Prism_S \simeq R \Gamma (S, \mathcal{O}^{\w{pris}}). $$

Further, we note that for any prism $(A,I)$, there is a decreasing filtration $\mathcal{N}^{\ge \bullet} A$ called the Nygaard filtration which is defined as 

$$ \mathcal{N}^{\ge i}(A):= \varphi^{-1} (I^i).$$
This equips the sheaf of rings $\mathcal{O}^{\w{pris}}$ with a decreasing filtration $\mathcal{N}^{\ge \bullet} \mathcal{O}^{\w{pris}}$ of sheaves. \\

\begin{remark}
We explain the construction of the Nygaard filtration in detail. Using \cite[Cor. 3.3.10]{AL19}, we have a morphism of topoi:

$$v: \w{Shv}((R)_{\Prism}) \to \w{Shv}((R)_{\textbf{Qsyn}}), $$ where $(R)_{\Prism}$ denotes the absolute prismatic site as in Def. 3.1.4 loc. sit. On $(R)_{\Prism}$ one defines a sheaf of rings $\mathcal{O}_{\Prism}$ by sending a prism $(A,I) \to A.$ Using the notion of Nygaard filtration on a prism one can equip this sheaf of rings $\mathcal{O}_{\Prism}$ with a Nygaard filtration $\mathcal{N}^{\ge \bullet} \mathcal{O}_{\Prism}$ which sends a prism $(A,I) \to \mathcal{N}^{\ge \bullet} A \subseteq A.$ In order to prove that the presheaf $(A,I) \to \mathcal{N}^{\ge \bullet} A$ is indeed a sheaf, we note that $\mathcal{N}^{\ge i}(A)$ is the kernel of a map of presheaf of rings $((A,I) \to A ) \to ((A,I) \to A/I^i)$ obtained by composing $\varphi: A \to A$ with the projection map $A \to A/I^i.$ Therefore it will be enough to show that the presheaf $(A,I) \to A/I^i$ is a sheaf. This follows from the proof of \cite[Cor. 3.12.]{BS19} by noting that for a map $(A,I) \to (B,J)$ of prisms one has $I^i B = J^i$ by Lemma 3.5. loc. cit.     One defines $$\mathcal{N}^{\ge \bullet} \mathcal{O}^{\w{pris}}:= v_*\mathcal{N}^{\ge \bullet} \mathcal{O}_{\Prism} .$$
By \cite[Prop. 7.2.]{BS19} and the above definition, it follows that $\mathcal{N}^{\ge \bullet} \mathcal{O}^{\w{pris}}(S) \simeq \mathcal{N}^{\ge \bullet} \Prism_S$ for a QRSP ring $S \in (R)_{\textbf{Qsyn}}.$ In fact, by \cref{qr} this description can be used to define the sheaves $\mathcal{N}^{\ge \bullet}\mathcal{O}^{\w{pris}}$ after one proves that it forms a sheaf on the basis objects. However, a priori it is not obvious why this description has to define a sheaf on the basis objects. Actually, more is true: the functor that sends a QRSP algebra $S \to \mathcal{N}^{\ge \bullet}\Prism_S$ forms a sheaf with vanishing higher cohomology, i.e., $H^i (S, \mathcal{N}^{\ge \bullet} \mathcal{O}^{\w{pris}}) = 0$ for $i \ge 1$ and a QRSP algebra $S.$ This fact follows from \cite[Thm. 12.2.]{BS19} and \cite[Thm. 3.1.]{BMS19}

\end{remark}{}

\begin{definition}\label{c1}We call $\mathcal{Y}$ a quasi-syntomic stack over $R$ if it is a stack with respect to the site $R_{\mathbf{Qsyn}}.$

\end{definition}{}

\begin{remark}If $A \in R_{\textbf{Qsyn}}$ then by fpqc descent it follows that $h_{A}$ is a sheaf with respect to $R_{\textbf{Qsyn}}.$ If $\mathcal{Y}$ is a quasi-syntomic stack, an arrow $h_A \to \mathcal{Y}$ will be denoted as $\w{Spf}\,A \to \mathcal{Y}.$

\end{remark}{}

\begin{remark}\label{c4}Let $X$ be a $p$-adic formal scheme over $R.$ Then setting $\underline{X}(A):= \w{Hom}_{\w{formal sch}}(\w{Spf}\, A, X)$ defines a quasi-syntomic sheaf on $R_{\mathbf{Qsyn}}$ which we will denote by $\underline{X}.$

\end{remark}{}

\begin{definition}\label{keno1}Let $\mathcal{Y}$ be a quasi-syntomic stack. We define the \textit{prismatic cohomology of the quasi-syntomic stack} $\mathcal{Y}$ to be 

$$R \Gamma_{\Prism} (\mathcal{Y}) := R\varprojlim_{\w{Spf}\,A \to \mathcal{Y} } \Prism_A,$$ where $A \in R_{\textbf{Qsyn}}$ and the derived limit is taken in the derived category of abelian groups.

\end{definition}{}

\begin{remark}\label{c2}
Let $\mathcal{Y}$ be a quasi-syntomic stack. We give an alternative description of $R \Gamma_{\Prism}(\mathcal{Y}).$ One can define the quasi-syntomic site of $\mathcal{Y}$ denoted as $\mathcal{Y}_{\textbf{Qsyn}}$ to be the category whose objects are $\w{Spf}\, A \to \mathcal{Y}$ for $A \in R_{\textbf{Qsyn}}$ equipped with the quasi-syntomic covers, i.e., we let $\mathcal{Y}_{\textbf{Qsyn}}:= R_{{\textbf{Qsyn}}_{/ \mathcal{Y}}}.$ Then by the discussion preceding \cref{c1}, we obtain a sheaf of rings $\mathcal{O}^{\w{pris}}$ on $\mathcal{Y}_{\textbf{Qsyn}}.$ Then it follows that $R\Gamma_{\Prism}(\mathcal{Y}) \simeq R \Gamma (\mathcal{Y}, \mathcal{O}^{\w{pris}}).$
\end{remark}
\begin{definition}\label{nag}Let $\mathcal{Y}$ be a quasi-syntomic stack. The \textit{Nygaard filtration} on prismatic cohomology $R\Gamma_{\Prism}(\mathcal{Y})$ is defined to be 

$$\mathcal{N}^{\ge \bullet}R \Gamma_{\Prism}(\mathcal{Y}):= R \Gamma(\mathcal{Y}, \mathcal{N}^{\ge \bullet} \mathcal{O}^{\w{pris}}). $$

\end{definition}{}

\begin{definition}
Now we consider an algebraic stack $\mathcal{X}$ over $R.$ We will define the $p$-adic completion of $\mathcal{X}$ as a quasi-syntomic stack which will be denoted as $\widehat{\mathcal{X}}.$ We define $$\widehat{\mathcal{X}}(A):= \varprojlim \mathcal{X}(A/p^n),$$ where $A \in R_{\textbf{Qsyn}}.$
This defines a sheaf of groupoids in the site ${R}_{\textbf{Qsyn}}.$ Indeed, $\widehat{\mathcal{X}}$ by definition is an inverse limit of the sheaves on ${R}_{\textbf{Qsyn}}$ which sends $A \to \mathcal{X}(A/p^n)$ and therefore has to be a sheaf. 



\end{definition}

\begin{definition}We define the prismatic cohomology of a stack $\mathcal{X}$ to be 

$$R \Gamma_{\Prism}(\mathcal{X}):= R\Gamma_{\Prism}(\widehat{\mathcal X}_{\textbf{}}). $$

\end{definition}{}

\begin{proposition}\label{pri1}If $X$ is a $p$-adic formal scheme over $R$, then $R \Gamma_{\Prism} (X) \simeq R \text{Hom}_{D(R_{\textbf{Qsyn}})}(\mathbb{Z}[\underline{X}], \mathcal{O}^{\text{pris}}),$ where $\underline{X}$ is the quasi-syntomic sheaf associated to $X$ on $(R)_{\textbf{Qsyn}}.$

\end{proposition} 

\begin{proof} This follows from \cref{c4} and \cref{c2} by adjunction.
\end{proof}

\begin{remark}Instead of considering quasi-syntomic stacks, one could have also worked with ``$p$-adic formal stacks" which can be defined to be stacks with respect to the Grothendieck site on the category of $p$-complete and bounded $p^\infty$-torsion rings equipped with $p$-completely faithfully flat covers. Any $p$-adic formal stack can also be regarded as a quasi-syntomic stack. The reason we work with the notion of quasi-syntomic stacks is because the notion of prismatic cohomology in \cref{keno1} would ultimately regard a $p$-adic formal stack as a quasi-syntomic stack. Hence one might as well define prismatic cohomology for quasi-syntomic stacks.

\end{remark}{}


\subsection{Application to prismatic Dieudonn\'e theory}\label{secc4.2}

From now we will assume that the ring $R$ is a QRSP ring. Let $(\Prism_R,I)$ be the prism associated to $R$ by taking prismatic cohomology. We briefly recall some definitions from \cite{AL19} and review their theorem on classification of $p$-divisible groups in terms of filtered prismatic Dieudonn\'e modules.

\begin{definition}\cite[Def. 4.1.10.]{AL19} A filtered prismatic Dieudonn\'e module over $R$ is a collection $(M, \w{Fil} M, \varphi_M)$ consisting of a finite locally free $\Prism_R$-module $M$, a $\Prism_R$-submodule $\w{Fil} M$, and $\varphi$ linear map $\varphi_M:M \to M$ satisfying 

\noindent
(1) $\varphi_M(\w{Fil}M) \subset IM$ and  $\varphi_M(\w{Fil}M)$ generates $I M$ as a $\Prism_R$-module\\
\noindent
(2) $(\mathcal{N}^{\ge 1}\Prism_R) M \subset \w{Fil} M$ and $M/ \w{Fil} M$ is a finite locally free finite $R$-module.
\end{definition}{}

Now let $G$ be a $p$-divisible group. One makes the following definition.

\begin{definition}The quasi-syntomic sheaf $\underline{G} \in (R)_{\textbf{Qsyn}}$ associated to a $p$-divisible group $G$ is defined to be $\varinjlim \underline{G_n}$ where $\underline{G_n} \in (R)_{\textbf{Qsyn}}$ is the quasi-syntomic sheaf associated to $\text{Spf}\,G_n.$

\end{definition}

The authors in \cite{AL19} make the following definition which can be seen as a generalization of \cref{gen}.
\begin{definition}{\cite[Def. 4.2.8.]{AL19}} Let $G$ be a $p$-divisible group over $R.$ We define $$M_{\Prism}(G):=\w{Ext}^1_{R_{\mathbf{Qsyn}}}(\underline{G}, \mathcal{O}^{\w{pris}}), $$

$$\w{Fil}M_{\Prism}(G):= \w{Ext}^1_{R_{\mathbf{Qsyn}}}(\underline{G}, \mathcal{N}^{\ge 1 }\mathcal{O}^{\w{pris}}),$$
and $\varphi_{M_{\Prism}(G)}$ as the endomorphism induced by $\varphi$ on $\mathcal{O}^{\w{pris}}.$ Then $$\underline{M}_{\Prism}(G):=  ( M_{\Prism}(G),\w{Fil}M_{\Prism}(G),\varphi_{M_{\Prism}(G)}  ) $$ is called the \textit{filtered prismatic Dieudonn\'e module} of $G.$

\end{definition}{}

In this case, the main theorem of \cite{AL19} says the following.

\begin{theorem}The filtered prismatic Dieudonn\'e module functor $$G \to \underline{M}_{\Prism}(G) $$ defines an anti-equivalence between the category $p$-divisible groups over $R$ and the category of filtered prismatic Dieudonn\'e modules over $R.$

\end{theorem}{}

Now we proceed towards proving \cref{thmpri}. 
Let $G= \w{Spec}\, B$ be a finite flat group scheme over $R.$ Since $G$ is syntomic, it follows that $\text{Spf} (B)$ is quasi-syntomic. If ${BG}$ denotes the associated quasi-syntomic stack, then $* \to {BG}$ is a $\underline{G}$-torsor and is a quasi-syntomic cover. The \v{C}ech nerve of the map $* \to {BG}$ is given by the following simplicial quasi-syntomic sheaf

 \[\xymatrix{
 \cdots \underline{ G} \times \underline{ G} \times \underline{ G}  \ar[r]<4.5pt>\ar[r]<1.5pt>\ar[r]<-4.5pt>\ar[r]<-1.5pt> &  \underline{  G} \times \underline{ G}  \ar[r]<3pt>\ar[r]\ar[r]<-3pt>  &   \underline{ G}  \ar[r]<1.5pt>\ar[r]<-1.5pt> & \ *\ .
}\]
The associated simplicial abelian group object is   

\[\xymatrix{\cdots \mathbb{Z}[\underline{ G} \times \underline{ G} \times \underline{ G}]  \ar[r]<4.5pt>\ar[r]<1.5pt>\ar[r]<-4.5pt>\ar[r]<-1.5pt> &  \mathbb{Z}[\underline{  G} \times \underline{ G}]  \ar[r]<3pt>\ar[r]\ar[r]<-3pt>  &  \mathbb{Z} [\underline{ G}]  \ar[r]<1.5pt>\ar[r]<-1.5pt> &\ \mathbb{Z}\ .
}\]
With this simplicial object we can associate an object of $D(R_{\textbf{Qsyn}})$ which we will denote simply by $\mathbb{Z}[{BG}].$\\

\begin{proposition}\label{pri2}Let $G$ be a finite flat group scheme over $R.$ Then $R \Gamma_{\Prism} (BG) \simeq R\w{Hom}_{D(R_{\textbf{Qsyn}})}(\mathbb{Z}{[BG}], \mathcal{O}^{\w{pris}}).$

\end{proposition}{}
\begin{proof}By \v{C}ech descent along the quasi-syntomic cover $* \to \underline{BG}$ we obtain that 

       \[ R \Gamma_{\Prism}(BG) \simeq R \varprojlim \big(\cosimp {R \Gamma_{\Prism}(*)}{R \Gamma_{\Prism}(G)}{R \Gamma_{\Prism}(G \times G)} \big).\]
Which by \Cref{pri1} is  \[ \simeq R \varprojlim \big(\cosimp {R\w{Hom}_{D(R)}(\mathbb Z, \mathcal O^{\w{pris}})}{R\w{Hom}_{D(R)}(\mathbb Z[\underline{G} ],\mathcal O^{\w{pris}} )}{R\w{Hom}_{D(R))}(\mathbb Z [\underline{G}\times \underline{G}], \mathcal O^{\w{pris}})} \big) .\]
We can take the $R \varprojlim$ inside as a homotopy colimit, which gives us that the above is $\simeq R\w{Hom}_{D(R_{\textbf{Qsyn}})}(\mathbb{Z}[{BG}], \mathcal{O}^{\w{pris}}),$ as desired. \end{proof}{}

\begin{definition}\label{pdiv} The classifying stack of a $p$-divisible group $G= \left \{G_n \right \}$ is defined to be $BG := \varinjlim {B G_n}$ where the colimit is taken in the category of quasi-syntomic stacks with respect to the site $R_{\textbf{Qsyn}}.$
\end{definition}

Since filtered colimits commute with finite limits, the \v{C}ech nerve of $* \to BG$ is given by the simplicial quasi-syntomic sheaf 

  \[\xymatrix{
 \cdots \underline{ G} \times \underline{ G} \times \underline{ G}  \ar[r]<4.5pt>\ar[r]<1.5pt>\ar[r]<-4.5pt>\ar[r]<-1.5pt> &  \underline{  G} \times \underline{ G}  \ar[r]<3pt>\ar[r]\ar[r]<-3pt>  &   \underline{ G}  \ar[r]<1.5pt>\ar[r]<-1.5pt> & \ *\ .
}\]
The associated simplicial abelian group object is 

 \[\xymatrix{\cdots \mathbb{Z}[\underline{ G} \times \underline{ G} \times \underline{ G}]  \ar[r]<4.5pt>\ar[r]<1.5pt>\ar[r]<-4.5pt>\ar[r]<-1.5pt> &  \mathbb{Z}[\underline{  G} \times \underline{ G}]  \ar[r]<3pt>\ar[r]\ar[r]<-3pt>  &  \mathbb{Z} [\underline{ G}]  \ar[r]<1.5pt>\ar[r]<-1.5pt> &\ \mathbb{Z}\ .
}\]
With this simplicial object, we can associate an object of $D(R_{\textbf{Qsyn}})$ which will be denoted by $\mathbb{Z}[BG].$ Since filtered colimits are exact in the category of abelian sheaves on $R_{\textbf{Qsyn}},$ we have that $\varinjlim \mathbb{Z}[BG_n] \simeq \mathbb{Z}[BG].$\\
\begin{proposition}\label{pri3} For a $p$-divisible group $G= \left \{ G_n \right \}$ we have $R \Gamma_{\Prism}  (BG) \simeq \varprojlim R \Gamma_{\Prism} (B G_n).$
\end{proposition}

\begin{proof}This follows formally since $BG$ is defined to be filtered colimit of the quasi-syntomic stacks $BG_n.$ \end{proof}

\begin{proposition}\label{pri4}Let $G$ be a $p$-divisible group over $R.$ Then $R \Gamma_{\Prism} (BG) \simeq R\w{Hom}_{D(R_{\textbf{Qsyn}})}(\mathbb{Z}[\underline{BG}], \mathcal{O}^{\w{pris}}).$

\end{proposition}{}

\begin{proof}
By \cref{pri3} and \Cref{pri2} we have $$R \Gamma_{\Prism}  (BG) \simeq \varprojlim R \Gamma_{\Prism} (B G_n) \simeq \varprojlim R\w{Hom}_{D(R_{\textbf{Qsyn}})}(\mathbb{Z}[{BG_n}], \mathcal{O}^{\w{pris}}).$$
Now $$ \varprojlim R\w{Hom}_{D(R_{\textbf{Qsyn}})}(\mathbb{Z}[{BG_n}], \mathcal{O}^{\w{pris}}) \simeq R\w{Hom}_{D(R_{\textbf{Qsyn}})}(\varinjlim\mathbb{Z}[{BG_n}], \mathcal{O}^{\w{pris}}).$$\\
Since $\varinjlim \mathbb{Z}[BG_n] \simeq \mathbb{Z}[BG],$ we obtain the required statement.
\end{proof}{}

\begin{remark}Alternatively, one could descend along the effective epimorphism $* \to BG.$

\end{remark}{}

\begin{proposition}\label{pri5}
There is a spectral sequence with $E_2$-page $$E_2^{i,j}= \w{Ext}_{R_{\textbf{Qsyn}}}^i (H^{-j}(\mathbb{Z}[B G], \mathcal{O}^{pris})\implies H^{i+j}_{\Prism}(BG), $$ and another spectral sequence with $E_1$-page 
$$ E_1^{i,j}=H^j_{\Prism}(G^i) \implies H^{i+j}_{\Prism}(BG), $$ where $G^i$ denotes the $i$ fold fibre product of $\text{Spf}\, G$ with itself. By convention, $G^0 = *.$

\end{proposition}{}

\begin{proof}This follows in a way similar to \Cref{lol3}.
\end{proof}{}

\begin{lemma}\label{pri6}$H^0 (\mathbb{Z}[BG]) \simeq \mathbb{Z},$ $H^{-1}(\mathbb{Z}[BG]) \simeq \underline{G},$ and $H^{-2}(\mathbb{Z}[BG]) \simeq \underline{G} \wedge \underline{G}.$

\end{lemma}{}
\begin{proof}
This follows exactly in the same ways as in the proof of \Cref{hah} after further noting that for an abelian group $G$, the second group homology with integral coefficients $H_2 (G, \mathbb{Z}) \simeq G \wedge G.$ This isomorphism can be found in \cite[Thm. 3.]{CM52}.
\end{proof}
\begin{proposition}\label{pri7}
There is a natural isomorphism $H^2_{\Prism}(BG) \simeq \w{Ext}^1_{R_{\textbf{Qsyn}}} (\underline{G}, \mathcal{O}^{\w{pris}}).$
\end{proposition}{}

\begin{proof}By the $E_2$ spectral sequence from \Cref{pri5} and \Cref{pri6}, this will follow once we prove that $\w{Hom}_{R_{\mathbf{Qsyn}}}(\underline{G} \wedge \underline{G}, \mathcal{O}^{\w{pris}})= 0$ and $\w{Ext}^i_{R_{\mathbf{Qsyn}}}(\mathbb{Z}, \mathcal{O}^{\w{pris}})=0$ for $i \ge 1$. We begin by proving the first vanishing. We have a surjection $\underline{G}\otimes \underline{G} \to \underline{G} \wedge \underline{G}\to 0$ which gives us an injection $$0 \to \w{Hom}_{R_{\mathbf{Qsyn}}}(\underline{G} \wedge \underline{G}, \mathcal{O}^{\w{pris}}) \to \w{Hom}_{R_{\mathbf{Qsyn}}}(\underline{G} \otimes \underline{G}, \mathcal{O}^{\w{pris}}). $$ Thus it is enough to show that $\w{Hom}_{R_{\mathbf{Qsyn}}}(\underline{G} \otimes \underline{G}, \mathcal{O}^{\w{pris}}) = 0.$ Indeed,

$$\w{Hom}_{R_{\mathbf{Qsyn}}}(\underline{G} \otimes \underline{G}, \mathcal{O}^{\w{pris}}) \simeq \w{Hom}_{R_{\mathbf{Qsyn}}}(\underline{G}, \mathscr{H}{om}_{R_{\mathbf{Qsyn}}}( \underline{G}, \mathcal{O}^{\w{pris}})),$$ 
and $\mathscr{H}{om}_{R_{\mathbf{Qsyn}}}( \underline{G}, \mathcal{O}^{\w{pris}})= 0$ since $G$ is $p$-divisible and $\mathcal{O}^{\w{pris}}$ is $p$-adically complete. Therefore, we obtain the required statement. Now the second vanishing follows from the fact that for a QRSP ring $S$, its prismatic cohomology $\Prism_S$ is discrete, i.e., lives only in degre 0. This finishes the proof. \end{proof}{}

\begin{proposition}\label{pri8}In the above situation, $\mathcal{N}^{\ge 1} H^2_{\Prism} (BG):= H^2 (\mathcal{N}^{\ge 1}R \Gamma_{\Prism}(BG)) \subset H^2_{\Prism} (BG)$ and $\mathcal{N}^{\ge 1} H^2_{\Prism} (BG) \simeq \w{Ext}^1_{R_{\textbf{Qsyn}}} (\underline{G}, \mathcal{N}^{\ge 1}\mathcal{O}^{\w{pris}})$

\end{proposition}{}

\begin{proof}By \cref{nag}, we have $\mathcal{N}^{\ge 1}R \Gamma_{\Prism}(BG) = R \Gamma (BG, \mathcal{N}^{\ge 1} \mathcal{O}^{\w{pris}}).$ Hence, analogous to \cref{pri2}, by descent along the effective epimorphism $* \to BG$, we obtain 

$$\mathcal{N}^{\ge 1}R \Gamma_{\Prism}(BG) \simeq R\w{Hom}_{D(R_{\textbf{Qsyn}})}(\mathbb{Z}[BG], \mathcal{N}^{\ge 1} \mathcal{O}^{\w{pris}}). $$
By the spectral sequence analogous to \cref{pri5}, we obtain that in order to prove $\mathcal{N}^{\ge 1} H^2_{\Prism} (BG) \simeq \w{Ext}^1_{R_{\textbf{Qsyn}}} (\underline{G}, \mathcal{N}^{\ge 1}\mathcal{O}^{\w{pris}})$, by \cref{pri6} it is enough to prove that $\w{Hom}_{R_{\textbf{Qsyn}}} (\underline{G} \wedge \underline{G}, \mathcal{N}^{\ge 1}\mathcal{O}^{\w{pris}})=0$ and $\w{Ext}^i_{R_{\textbf{Qsyn}}} (\mathbb{Z}, \mathcal{N}^{\ge 1}\mathcal{O}^{\w{pris}})=0$ for $i \ge 2$. By the proof of \cref{pri7}, in order to prove the first vanishing it is enough to prove that $\mathscr{H}om_{R_{\textbf{Qsyn}}}(\underline{G}, \mathcal{N}^{\ge 1}\mathcal{O}^{\w{pris}}))=0.$ By the injection $0 \to  \mathcal{N}^{\ge 1}\mathcal{O}^{\w{pris}} \to \mathcal{O}^{\w{pris}} $ of sheaves, the required vanishing follows from the fact that  $\mathscr{H}{om}_{R_{\mathbf{Qsyn}}}( \underline{G}, \mathcal{O}^{\w{pris}})= 0$ which was noted in \cref{pri7}. Now for proving $\w{Ext}^i_{R_{\textbf{Qsyn}}} (\mathbb{Z}, \mathcal{N}^{\ge 1}\mathcal{O}^{\w{pris}})=0$ for $i \ge 2$, using the exact sequence \cite[Prop. 4.1.2.]{AL19} $$0 \to \mathcal{N}^{\ge 1}\mathcal{O}^{\w{pris}} \to \mathcal{O}^{\w{pris}} \to \mathcal{O} \to 0,$$ where $\mathcal{O}$ denotes the structure sheaf on $R_{\textbf{Qsyn}},$ it is enough to show that $\w{Ext}^i_{R_{\textbf{Qsyn}}}(\mathbb{Z}, \mathcal{O})=0$ for $i \ge 1$ which follows from $p$-complete faithfully flat descent for bounded $p^\infty$-torsion rings \cite[Remark. 4.9.]{BMS19}. The inclusion $\mathcal{N}^{\ge 1} H^2_{\Prism} (BG) \subset H^2_{\Prism} (BG)$ now follows from the exact sequence above and the fact that $\w{Hom}^{}(\underline{G}, \mathcal{O})= 0$ since $G$ is $p$-divisible and $\mathcal{O}$ is $p$-adically complete.  \end{proof}{}

\vspace{1mm}
\begin{proof}[Proof of \cref{thmpri}] Follows from \cref{pri7} and \cref{pri8}. \end{proof}{}

\begin{remark}We note that $\w{Ext}^1(\mathbb{Z}, \mathcal{N}^{\ge 1} \mathcal{O}^{\w{pris}})=0$ as well. This follows from the exact sequence above and \cite[Thm. 3.4.6]{AL19}.

\end{remark}

\begin{remark}Similar to \cref{lol4}, it follows that for a $p$-divisible group $G$ over $R$, we have $H^1_{\Prism}(BG) \simeq 0. $

\end{remark}

\begin{remark}One can define prismatic cohomology for higher quasi-syntomic stacks as well. In particular one can talk about prismatic cohomology of the $n$-stacks $K(G,n)$ for a $p$-divisible group $G.$ Similar to \cref{higher}, one can prove that $H^{n+1}_{\Prism}(BG) \simeq M_{\Prism}(G)$ and $H^n_{\Prism}(BG) = 0$ for $1 \le i \le n.$

\end{remark}{}

Now we look at the case where $\widehat{A}$ is the $p$-adic completion of an abelian scheme $A$ over $R.$ In this situation one can consider the $p$-divisible group associated to $\widehat{A}$ which will be written as $ {A}[p^\infty]. $ We let $B\widehat{A}$ denote the classifying stack of $\widehat{A}$ which we define to be the $p$-adic completion of the stack $BA.$ The quasi-syntomic sheaf represented by $\widehat{A}$ will simply be written as $\underline{A}$. Then 

\begin{proposition}\label{pri81} We have $H^2 _{\Prism} (B\widehat{A}) \simeq \w{Ext}^1_{R_{\textbf{Qsyn}}}(\underline{{A}}, \mathcal{O}^{\w{pris}})$ and $\mathcal{N}^{\ge 1} H^2_{\Prism} (B \widehat{A}) \simeq \w{Ext}^1_{R_\textbf{QSyn}}(\underline{A}, \mathcal{N}^{\ge 1} \mathcal{O}^{\w{pris}}  ).$  

\end{proposition}{}

\begin{proof}This follows exactly in the same way as the proof of \Cref{pri7} after noting that $\w{Ext}^i_{R_\textbf{Qsyn}}(\mathbb{Z}, \mathcal{O}^{\w{pris}})=0$ for $i \ge 1$ as before and $\mathscr{H}om_{R_{\textbf{Qsyn}}}(\underline{A}, \mathcal{O}^{\w{pris}})= 0$ by \cite[Thm. 4.5.6.]{AL19}. The second part follows. \end{proof}{}

\begin{proposition}\label{pri9}$H^2_{\Prism} (B \widehat{A}) $ is naturally isomorphic to the prismatic Dieudonn\'e module associated to the $p$-divisible group ${A}[p^\infty].$ Further, $\mathcal{N}^{\ge 1} H^2_{\Prism} (B \widehat{A}) \simeq \w{Fil} M_{\Prism} ( A[p^\infty]).$

\end{proposition}{}

\begin{proof}Using \Cref{pri81}, it is enough to prove that $\w{Ext}^1_{R_{\textbf{Qsyn}}}(\underline{A}, \mathcal{O}^{\w{pris}})$ is isomorphic to $\w{Ext}^1_{R_{\textbf{Qsyn}}}({A}[p^\infty], \mathcal{O}^{\w{pris}}).$ This follows from the isomorphisms  $$\mathscr{E}xt^1_{R_{\textbf{Qsyn}}}({A}[p^\infty],  \mathcal{O}^{\w{pris}}) \simeq \mathscr{E}xt^1_{R_{\textbf{Qsyn}}}(\underline{A},  \mathcal{O}^{\w{pris}}) $$and $$ \mathscr{E}xt^1_{R_{\textbf{Qsyn}}}({A}[p^\infty], \mathcal{N}^{\ge 1} \mathcal{O}^{\w{pris}}) \simeq \mathscr{E}xt^1_{R_{\textbf{Qsyn}}}(\underline{A}, \mathcal{N}^{\ge 1} \mathcal{O}^{\w{pris}})$$ after applying the global section functor and noting that $\mathscr{H}om_{R_{\textbf{Qsyn}}}({A}[p^\infty],  \mathcal{O}^{\w{pris}}) \simeq 0$ and $ \mathscr{H}om_{R_{\textbf{Qsyn}}}(\underline{A},  \mathcal{O}^{\w{pris}}) \simeq 0.$ \end{proof}{}

As in \cref{ab}, if $\widehat{A}$ is the $p$-adic completion of an abelian scheme $A$ over $R$, we can explicitly compute $H^* _{\Prism}(B (\widehat{A})).$ 

\begin{proposition}\label{pri10}We have a natural isomorphism $H^{2*}_{\Prism}(B \widehat{A}) \cong \w{Sym}^*(H^1_{\Prism}(\widehat{A}))$, and $H^i_{\Prism}(B (\widehat{A}))= 0$ for odd $i.$

\end{proposition}

\begin{proof}This follows exactly in the same way as in the proof in the crystalline case in \Cref{ab} from the $E_1$ spectral sequence in \cref{pri5}, K\"unneth formula in prismatic cohomology \cite[Corollary 3.5.2.]{AL19} and the calculations of Prismatic cohomology of abelian varieties in Corollary 4.5.8. loc. cit. 
\end{proof}{}

\end{document}